\documentclass[a4paper,oneside,12pt]{amsart}
\usepackage{graphicx,amssymb,amsthm,amsmath}
\newtheorem{thm}{Theorem}
\newtheorem*{thm*}{Theorem}

\newtheorem{lem}{Lemma}
\newtheorem{prop}{Proposition}
\newtheorem*{prop*}{Proposition}
\newtheorem{rem}{Remark}

\newcommand{\ninf}{n\rightarrow\infty}

\renewcommand{\l}{\left}
\renewcommand{\r}{\right}

\begin{document}
\title[Zeroes of the spectral density]
{Zeroes of the spectral density of the
\\
periodic Schr\"odinger operator with
\\
Wigner-von Neumann potential}
\author{Sergey Naboko}
    \address{Department of Mathematical Physics, Institute of Physics, St. Petersburg University, Ulianovskaia 1, St. Petergoff, St. Petersburg, Russia, 198904}
    \email{sergey.naboko@gmail.com}
\author{Sergey Simonov}
    \address{Chebyshev Laboratory, Department of Mathematics and Mechanics,
    Saint-Petersburg State University
    14th Line, 29b, Saint-Petersburg, 199178 Russia}
    \email{sergey\_simonov@mail.ru}
\date{}

\begin{abstract}
We consider the Schr\"odinger operator $\mathcal L_{\alpha}$ on
the half-line with a periodic background potential and the
Wigner-von Neumann potential of Coulomb type: \linebreak
$\frac{c\sin(2\omega x+\delta)}{x+1}$. It is known that the
continuous spectrum of the operator $\mathcal L_{\alpha}$ has the
same band-gap structure as the free periodic operator, whereas in
each band of the absolutely continuous spectrum there exist two
points (so-called critical or resonance) where the operator
$\mathcal L_{\alpha}$ has a subordinate solution, which can be
either an eigenvalue or a ``half-bound'' state. The phenomenon of
an embedded eigenvalue is unstable under the change of the
boundary condition as well as under the local change of the
potential, in other words, it is not generic. We prove that in the
general case the spectral density of the operator $\mathcal
L_{\alpha}$ has power-like zeroes at critical points (i.e., the
absolutely continuous spectrum has pseudogaps). This phenomenon is
stable in the above-mentioned sense.
\end{abstract}

\maketitle

\section{Introduction}
It is well-known that the Schr\"odinger operator on the half-line
with a summable potential has purely absolutely continuous
spectrum on $\mathbb R_+$ \cite{Titchmarsh-1946-1}. It is also
known \cite{Deift-Killip-1999} that if the potential is square
summable, then the absolutely continuous spectrum still covers the
positive half-line, see also the work \cite{Kiselev-1996}. In this
case however there is no guarantee that the positive singular
spectrum is empty. Moreover, there are explicit examples of
potentials \cite{Naboko-1986,Simon-1997}, which produce dense
point spectrum on $\mathbb R_+$. In a sense, the simplest example
of the potential producing positive eigenvalues is the so-called
Wigner-von Neumann potential $\frac{c\sin(2\omega x)}x$. The
self-adjoint operator $\mathcal L_{0,\alpha}$ given by the
differential expression
$l_0:=-\frac{d^2}{dx^2}+\frac{c\sin(2\omega x)}x$ and the boundary
condition $\psi(0)\cos\alpha-\psi'(0)\sin\alpha=0$, acting in
$L_2(\mathbb R_+)$ on the domain
\begin{equation*}
    \text{Dom}\,\mathcal L_{0,\alpha}=\biggl\{\psi\in L_2(\mathbb R_+)\cap H^2_{loc}(\mathbb R_+): l_0\psi\in L_2(\mathbb R_+), \psi(0)\cos\alpha-\psi'(0)\sin\alpha=0\biggr\}
\end{equation*}
may have an eigenvalue at the point $\omega^2$
\cite{Wigner-von-Neumann-1929}, while the rest of $\mathbb R_+$ is
covered by the purely absolutely continuous spectrum
\cite{Behncke-1991-I}. Due to the subordinacy theory
\cite{Gilbert-Pearson-1987}, this can be seen from the behavior of
generalized eigenvectors (i.e., the solutions of the spectral
equation $l_0\psi=\lambda\psi$): for $\lambda\in\mathbb
R_+\backslash\{\omega^2\}$ there is a base of them with
asymptotics $e^{\pm i\sqrt\lambda x}+o(1)$ as
$x\rightarrow+\infty$. For $\lambda=\omega^2$ the base is
different: $x^{\frac{c}{4\omega}}(\sin(\omega x)+o(1))$ and
$x^{-\frac{c}{4\omega}}(\cos(\omega x)+o(1))$. This point is
called the resonance (or critical) point due to the change of the
type of asymptotics of generalized eigenvectors. One of the
solutions (up to a constant) of the equation
$l_0\psi=\omega^2\psi$ is subordinate. If on top of that it
belongs to $L_2(\mathbb R_+)$ and satisfies the boundary condition
at the origin, then it represents an eigenfunction of $\mathcal
L_{0,\alpha}$. It is clear that the effect of appearance of
positive eigenvalue is highly unstable. It happens with
probability one in a suitable sense: the eigenvalue disappears if
one changes slightly  the boundary condition or adds a summable
(or even compactly supported) perturbation  to the potential. If
the notion of resonances makes sense for the operator (i.e., if
the kernel of the resolvent admits the analytic continuation
through the continuous spectrum), this eigenvalue may become a
resonance under such perturbation. The set of summable functions
which one could utilize as perturbations of the potential
preserving the eigenvalue at the point $\omega^2$ was studied in
\cite{Herbst-2002}. On the other hand, the phenomenon of the
change of the type of asymptotics of generalized eigenvectors at a
point of the absolutely continuous spectrum is stable under
summable perturbations of the potential and does not depend on the
boundary condition. It is therefore meaningful to consider objects
which are a more or less stable in this sense: one can study the
Weyl-Titchmarsh function $m$ or the spectral density $\rho'$ (the
derivative of the spectral function of the operator), which are
related by the equality
\begin{equation*}
    \rho'(\lambda)=\frac1{\pi}\,\text{Im}\,m(\lambda+i0)\text{ for a.a. }\lambda\in\mathbb R.
\end{equation*}
The  behavior of the Weyl function near resonance points has been
studied by Hinton-Klaus-Shaw
\cite{Hinton-Klaus-Shaw-1991,Klaus-1991} and Behncke
\cite{Behncke-1991-I,Behncke-1991-II,Behncke-1994} for
Schr\"odinger and Dirac operators with potentials of, in
particular, Wigner-von Neumann type.

In the present paper we consider the  differential Schr\"odinger
operator $\mathcal L_{\alpha}$ with the potential that is the sum
of the following three parts: a periodic background $q$, a
potential of Wigner-von Neumann type $\frac{c\sin(2\omega
x+\delta)}{x+1}$ and a summable part $q_1$. The operator is
defined by the differential expression
\begin{equation}\label{l}
    l:=-\frac{d^2}{dx^2}
    +q(x)+\frac{c\sin(2\omega x+\delta)}{x+1}+q_1(x)
\end{equation}
and the boundary  condition
$\psi(0)\cos\alpha-\psi'(0)\sin\alpha=0$, $\alpha\in[0;\pi)$. It
acts in $L_2(\mathbb R_+)$ on a suitable domain. We assume that
the function $q$ has the period $a$ and is summable over this
period: $q\in L_1(0;a)$. The parameters $c$ and $\delta$ are real
constants. The operator $\mathcal L_{\alpha}$ is then self-adjoint
in $L_2(\mathbb R_+)$.

The geometry  of the spectrum in the case of periodic background
differs from the case $q(x)\equiv0$. It is well-known (see, e.g.,
\cite{Kurasov-Naboko-2007}) that the absolutely continuous
spectrum of the operator given by the expression $l$ on the whole
real line coincides with the spectrum of the corresponding
periodic operator on the whole real line,
\begin{equation}\label{L-per}
    \mathcal L_{per}=-\frac{d^2}{dx^2}+q(x),
\end{equation}
i.e., has a band-gap structure:
\begin{equation*}
    \sigma(\mathcal L_{per})=:\bigcup\limits_{j=0}^{\infty}([\lambda_{2j};\mu_{2j}]\cup[\mu_{2j+1};\lambda_{2j+1}]),
\end{equation*}
where
$\lambda_0<\mu_0\le\mu_1<\lambda_1\le\lambda_2<\mu_2\le\mu_3<\lambda_3\le\lambda_4<...$
It is clear that the absolutely continuous spectrum of $\mathcal
L_{\alpha}$ coincides set-wise with $\sigma(\mathcal L_{per})$,
although $\sigma(\mathcal L_{per})$ has multiplicity two, whereas
 $\sigma(\mathcal L_{\alpha})$ is simple; note that $\mathcal
L_{\alpha}$ is considered on the half-line. In every band there
exist two resonance points $\nu_{j,+}$ and $\nu_{j,-}$. The type
of asymptotics of generalized eigenvectors at these points is
different from that in other points of the absolutely continuous
spectrum. Subordinate solution appears and thus each of the
resonance points can be an eigenvalue of the operator $\mathcal
L_{\alpha}$ as long as this solution satisfies the boundary
condition and belongs to $L_2(\mathbb R_+)$. The exact locations
of the points $\nu_{j,\pm}$ are determined by the "quantization
conditions" \cite{Kurasov-Naboko-2007}
\begin{equation*}
    k(\nu_{j,+})=\pi\left(j+1-\left\{\frac{a\omega}{\pi}\right\}\right),
    \
    k(\nu_{j,-})=\pi\left(j+\left\{\frac{a\omega}{\pi}\right\}\right),
    \
    j\ge0,
\end{equation*}
where $k(\lambda)$ is the quasi-momentum of the periodic operator
$\mathcal L_{per}$  and $\{\cdot\}$ is the standard fractional part
function. The above-mentioned condition
$\omega\notin\frac{\pi\mathbb Z}{2a}$ guarantees that resonance
points firstly do not coincide with the band boundaries and
secondly that they do not glue up together.

The main  goal of the present paper is to analyze the interplay
between the periodic structure and the ``singular'' Wigner-von
Neumann perturbation. The main result is Theorem \ref{thm result}
of Section \ref{section result}, which essentially tells us that
the spectral density of the operator $\mathcal L_{\alpha}$ has
power-like zeroes at each of the resonance points. Let
$\psi_+(x,\lambda)$ and $\psi_-(x,\lambda)$ be the Bloch solutions
of the periodic equation $-\psi''(x)+q(x)\psi(x)=\lambda\psi(x)$.
Let $\varphi_{\alpha}(x,\lambda)$ be the solution of the Cauchy
problem
\begin{equation}\label{phi}
    l\varphi_{\alpha}=\lambda\varphi_{\alpha},\ \varphi_{\alpha}(0)=\sin\alpha,\ \varphi'_{\alpha}(0)=\cos\alpha.
\end{equation}
Denote by $W\{\psi_+(\lambda),\psi_-(\lambda)\}=W\{\psi_+(\cdot,\lambda),\psi_-(\cdot,\lambda)\}$ the Wronskian of two Bloch solutions.

    \begin{thm*}
    Let $q_1\in L_1(\mathbb R_+)$, $\omega\notin\frac{\pi\mathbb Z}{2a}$, and $\rho'_{\alpha}(\lambda)$ be the spectral  density of the operator $\mathcal L_{\alpha}$. Let the index $j\ge0$. If $\alpha$ is such that the solution $\varphi_{\alpha}(x,\nu_{j,+})$ of \eqref{phi} is not a subordinate one, then there exist two non-zero limits
    \begin{equation*}
        \lim_{\lambda\rightarrow\nu_{j,+}\pm0}\frac{\rho'_{\alpha}(\lambda)}
        {|\lambda-\nu_{j,+}|
        ^
        {\frac{2|c|}{a|W\{\psi_+(\nu_{j,+}),\psi_-(\nu_{j,+})\}|}
        \l|\int\limits_0^a\psi^2_+(t,\nu_{j,+})e^{2i\omega t}dt\r|}}.
    \end{equation*}
    Analogously, if $\alpha$ is such that $\varphi_{\alpha}(x,\nu_{j,-})$ is not subordinate, then there exist two non-zero limits
    \begin{equation*}
        \lim_{\lambda\rightarrow\nu_{j,-}\pm0}\frac{\rho'_{\alpha}(\lambda)}{|\lambda-\nu_{j,-}|^{\frac{2|c|}
        {a|W\{\psi_+(\nu_{j,+}),\psi_-(\nu_{j,+})\}|}
        \l|\int\limits_0^a\psi^2_-(t,\nu_{j,-})e^{2i\omega t}dt\r|}}.
    \end{equation*}
    \end{thm*}

Note that the  condition of non-subordinacy of $\varphi_{\alpha}$
is indeed satisfied in the generic case. Consider one of the
critical points $\nu_{cr}$. There is a unique $\alpha\in[0;\pi)$
for which $\varphi_{\alpha}$ is subordinate, denote it by
$\alpha_{cr}$. Denote also $\beta:=\frac{|c|}
{a|W\{\psi_+(\nu_{cr}),\psi_-(\nu_{cr})\}|}
\l|\int\limits_0^a\psi^2_{\pm}(t,\nu_{cr})e^{2i\omega t}dt\r|$
where the choice of sign coincides with the one in
$\nu_{cr}=\nu_{j,\pm}$. It follows from \cite{Kurasov-Naboko-2007}
and the analysis below that
$\varphi_{\alpha_{cr}}(x,\nu_{cr})=x^{-\beta}(c_+\psi_+(x,\nu_{cr})+c_-\psi_-(x,\nu_{cr})+o(1))$
as $x\rightarrow+\infty$ where $\psi_+,\psi_-$ are the Bloch
solutions of the periodic equation
$-\psi''(x)+q(x)\psi(x)=\lambda\psi(x)$ and $c_+,c_-$ are some
constants. Therefore $\nu_{cr}$ is an eigenvalue of $\mathcal
L_{\alpha_{cr}}$ iff $\beta>\frac12$. At the same time, for
$\beta>\frac12$ and $\alpha\neq\alpha_{cr}$ the order of the zero
of $\rho'_{\alpha}(\lambda)$ at the point $\nu_{cr}$ is greater
than one. This fact is in exact correspondence with the result of
Aronszajn-Donoghue \cite{Aronszajn-1957}, by which if the point
$\nu$ is an eigenvalue of the operator $\mathcal L_{\alpha_{cr}}$,
then the integral $\int_{\mathbb
R}\frac{d\rho_{\alpha}(\lambda)}{(\lambda-\nu)^2}$ is finite for
every $\alpha\neq\alpha_{cr}$.

It should be mentioned, that  this result in the particular case
of zero background periodic potential $q$ follows from the work of
Hinton-Klaus-Shaw \cite{Hinton-Klaus-Shaw-1991}. They considered
the differential Schr\"odinger operator on the half-line with the
potential, which is an infinite sum of terms of Wigner-von Neumann
type plus a rapidly decreasing term analogous to our $q_1$ but
decaying faster. They studied the behavior of the Weyl-Titchmarsh
function near the critical points in the case $\alpha=\alpha_{cr}$
and $\beta<\frac12$, i.e., when the solution $\varphi_{\alpha}$ is
subordinate but is not an eigenfunction (the so-called half-bound
state). The corresponding result for the case
$\alpha\neq\alpha_{cr}$ and any $\beta>0$ was a by-product of
their analysis. Nevertheless, the most essential part of the
problem in their work is similar to the one in our case. However,
our approach is different from that of papers
\cite{Hinton-Klaus-Shaw-1991,Klaus-1991}, see the detailed
discussion in Section \ref{section discussion}. Our goal was to
elaborate a more general approach (reduction to a model problem),
which would enable us to apply it without any significant changes
to different operators.

Zeroes of the  spectral density of the Schr\"odinger operator
divide the absolutely continuous spectrum of the operator into
independent parts. This phenomenon is called a pseudogap and has a
clear physical meaning. Eigenvalues embedded into the continuous
spectrum have been observed in experiment
\cite{Capasso-et-al-1992}. Operators with Wigner-von Neumann
potentials attracted attention of many other authors, e.g.,
\cite{Buslaev-Matveev-1970,Buslaev-Skriganov-1974,Matveev-1973,Hinton-Klaus-Shaw-1991,Klaus-1991,Behncke-1994,Behncke-1991-II,
Behncke-1991-I,Kurasov-1996,Kurasov-1992,Kurasov-Naboko-2007,Nesterov-2007}.

The paper is organized  as follows. In Section \ref{section
preliminaries}, we state the Weyl-Titchmarsh type formula for the
operator $\mathcal L_{\alpha}$ which was proved in
\cite{Kurasov-Simonov-2011} and serves as our main tool in the
work with the spectral density. In Section \ref{section
discretization}, we write the spectral equation for the operator
$\mathcal L_{\alpha}$ in an equivalent form of a discrete linear
system and restate the Weyl-Titchmarsh type formula in terms of
the asymptotic behavior of solutions of that system. In Section
\ref{section model problem}, we study this discrete system (which
can be regarded as a model problem). This can be viewed as an
independent task. In Section \ref{section result}, we use the
results of the preceding analysis to find the asymptotics of the
spectral density of the operator $\mathcal L_{\alpha}$ near the
critical points. In Section \ref{section discussion}, we make
final comments characterizing our method in comparison to the one
utilized in the work \cite{Hinton-Klaus-Shaw-1991} and give a discussion of a few examples of its applicability.

\section{Preliminaries}\label{section preliminaries}
Spectral properties of second-order differential Schr\"odinger operators are related to the asymptotic behavior of their generalized eigenvectors \cite{Gilbert-Pearson-1987}. One of illustrations of this relation is the Weyl-Titchmarsh type formula for the spectral density in terms of asymptotic coefficients of the solution to the spectral equation that satisfies the boundary condition. The spectral equation $l\psi=\lambda\psi$ can be considered as a perturbation of the periodic equation $-\psi''(x)+q(x)\psi(x)=\lambda\psi(x)$.

Let us denote by $\partial:=\{\lambda_j,\mu_j,\ j\ge0\}$ the "generalized boundary" of
the spectrum of the periodic operator $\mathcal L_{per}$ (it can be not exactly the boundary of the set $\sigma(\mathcal L_{per})$, because the endpoints of the neighboring bands can coincide). We will use the following result (Weyl-Titchmarsh type formula).

    \begin{prop}[\cite{Kurasov-Simonov-2011}]\label{prop Weyl-Titchmarsh type formula}
    Let $\omega\notin\frac{\pi\mathbb Z}{2a}$ and $q_1\in L_1(\mathbb R_+)$. Then for every fixed
    \linebreak
    $\lambda\in\sigma(\mathcal L_{per})\backslash(\partial\cup\{\nu_{j,+},\nu_{j,-},j\ge0\})$ there exists a non-zero constant $A_{\alpha}(\lambda)$ depending on $\lambda$ such that
    \begin{equation}\label{asymptotics of phi-alpha}
    \begin{array}{l}
        \varphi_{\alpha}(x,\lambda)=A_{\alpha}(\lambda)\psi_-(x,\lambda)+
        \overline{A_{\alpha}(\lambda)}\psi_+(x,\lambda)+o(1)\text{ as }x\rightarrow+\infty,
        \\
        \varphi_{\alpha}'(x,\lambda)=A_{\alpha}(\lambda)\psi_-'(x,\lambda)+
        \overline{A_{\alpha}(\lambda)}\psi_+'(x,\lambda)+o(1)\text{ as }x\rightarrow+\infty,
    \end{array}
    \end{equation}
    and the following equality holds:
    \begin{equation}\label{Weyl-Titchmatsh type formula phi}
        \rho'_{\alpha}(\lambda)=\frac1{2\pi|W\{\psi_+(\lambda),\psi_-(\lambda)\}| \; |A_{\alpha}(\lambda)|^2},
    \end{equation}
    where $\rho'_{\alpha}(\lambda)$ is the spectral density of the operator $\mathcal L_{\alpha}$.
    \end{prop}

This result is a generalization of the classical Weyl-Titchmarsh (or Kodaira) formula \cite{Titchmarsh-1946-2,Kodaira-1949}. A variant of this formula for the Schr\"odinger operator with Wigner-von Neumann potential without the periodic background ($q(x)\equiv0$) follows from the results of \cite{Matveev-1973,Brown-Eastham-McCormack-1998}. In the case of discrete Schr\"odinger operator with Wigner-von Neumann potential an analogous formula is also known, see \cite{Damanik-Simon-2006,Janas-Simonov-2010}.

\section{Discretization}\label{section discretization}
In this section we pass over from the spectral equation $l\psi=\lambda\psi$ to a specially constructed discrete linear system essentially equivalent to $l\psi=\lambda\psi$ and rewrite the Weyl-Titchmarsh type formula expressing the spectral density in terms of solutions of this system.

\subsection*{Step 1. Variation of parameters and discretization: perturbation of the monodromy matrix}
First we perform a transformation of the spectral equation to the differential system in order to "get rid" of the periodic background potential (in fact, this simply is a variation of parameters). Consider $\lambda\in\sigma(\mathcal L_{per})\backslash\partial$ and equation $l\psi=\lambda\psi$. Define new vector-valued function $\eta(x)$ by the following equality:
\begin{equation*}
    \left(%
    \begin{array}{c}
    \psi(x) \\
    \psi'(x) \\
    \end{array}%
    \right)
    =
    \left(%
    \begin{array}{cc}
    \psi_-(x,\lambda) & \psi_+(x,\lambda) \\
    \psi_-'(x,\lambda) & \psi_+'(x,\lambda) \\
    \end{array}%
    \right)
    \eta(x).
\end{equation*}
By a straightforward substitution one has:
\begin{equation}\label{system for eta}
    \eta'(x)=L(x,\lambda)\eta(x),
\end{equation}
where
\begin{equation*}
    L(x,\lambda):=
    \frac{\frac{c\sin(2\omega x+\delta)}{x+1}+q_1(x)}{W\{\psi_+(\lambda),\psi_-(\lambda)\}}
    \left(%
    \begin{array}{cc}
    -\psi_+(x,\lambda)\psi_-(x,\lambda) & -\psi_+^2(x,\lambda) \\
    \psi_-^2(x,\lambda) & \psi_+(x,\lambda)\psi_-(x,\lambda) \\
    \end{array}%
    \right).
\end{equation*}
Denote by $\Phi(x_0,x,\lambda)$ the fundamental matrix for the system \eqref{system for eta} (i.e., the matrix solution of \eqref{system for eta} satisfying $\Psi(x_0,x_0,\lambda)\equiv I$). Denote also by $M_n(\lambda)$ the monodromy matrix corresponding to the shift by the period $a$ of the background potential, $M_n(\lambda):=\Phi(a(n-1),an,\lambda)$. Instead of the solution $\eta(x)$ of the differential system \eqref{system for eta} consider the following sequence $\{w_n\}_{n=1}^{\infty}$ of vectors from $\mathbb C^2$:
\begin{equation*}
    w_n:=\eta(a(n-1)),
\end{equation*}
which obviously solves the discrete linear system $w_{n+1}=M_n(\lambda)w_n$. In what follows, we deal with this system  transforming it to a simpler form. Using the standard perturbation methods (see for example \cite{Kurasov-Naboko-2007}) we can determine the matrix $M_n(\lambda)$ up to a summable (in $n$) term.

We have passed from the continuous variable $x$ to the discrete variable $n$. Later, in Section \ref{section model problem} we will return to continuous variables considering a new one, $y$. There is no strict necessity in such a trick. One could think that the discretization determined by the period $a$ of the background potential $q$ helps to get rid of this potential. However, this is not completely true: this role is taken by the above variation of parameters transformation. One could do simpler: the coefficient matrix of the differential system \eqref{system for eta} is the sum of two terms. The first term has a factor $\frac1{x+1}$ multiplied by an infinite sum of the exponential terms coming from the Fourier decomposition of the periodic parts of Bloch solutions. Each of these terms can "resonate" (become constant in $x$) for certain values of $\lambda$. This happens exactly at the resonance points $\nu_{j,\pm}$. The second term is a summable matrix-valued function. Non-resonating exponential terms can be eliminated using the uniform Harris-Lutz transformation in a fashion of \cite{Kurasov-Simonov-2011}. This approach would lead to a simple differential model system equivalent to the equation $l\psi=\lambda\psi$. However, we use the approach of the discretization, which brings some extra technical difficulties and complicates the notation. The reason is that we obtain as a result a discrete model system, which can serve wider needs. In particular, it is possible to consider the discrete Schr\"odinger operator with the discrete Wigner-von Neumann potential which has two critical points on the interval $[-2;2]$ (covered by the absolutely coninuous spectrum of this operator) and to prove that its spectral density has zeroes of the power type at the critical points. See Discussion for the exact formulation of this result. We plan to prove it in a forthcoming paper.

Let us introduce the following notation for uniformly summable sequences. Let $R_n(\lambda)$ be $2\times2$ matrices for $n\in\mathbb N$ and $\lambda\in S$ (where $S$ is an arbitrary set). We write $\{R_n(\lambda)\}_{n=1}^{\infty}\in l^1(S)$, if there exists a sequence of positive numbers $\{r_n\}_{n=1}^{\infty}\in l^1$ such that for every $\lambda\in S$ and $n\in\mathbb N$, $\|R_n(\lambda)\|<r_n$. We start with a simple asymptotic formula for the monodromy matrix:

    \begin{lem}\label{lem M-n(lambda)}
    For every $\lambda\in\sigma(\mathcal L_{per})\backslash\partial$,
    \begin{multline*}
    M_n(\lambda)=I+\frac1{an}\int_{a(n-1)}^{an}
    \frac{c\sin(2\omega t+\delta)}{W\{\psi_+(\lambda),\psi_-(\lambda)\}}
    \\
    \times
    \left(%
    \begin{array}{cc}
    -\psi_+(t,\lambda)\psi_-(t,\lambda) & -\psi_+^2(t,\lambda) \\
    \psi_-^2(t,\lambda) & \psi_+(t,\lambda)\psi_-(t,\lambda) \\
    \end{array}%
    \right)dt
    +R_n^{(1)}(\lambda),
    \end{multline*}
    where $\{R_n^{(1)}(\lambda)\}_{n=1}^{\infty}\in l^1(K)$ for every compact set $K\subset\sigma(\mathcal L_{per})\backslash\partial$.
    \end{lem}

Note that it follows that for every $n\in\mathbb N$ the matrix $R_n^{(1)}(\lambda)$ depends continuously on $\lambda\in\sigma(\mathcal L_{per})\backslash\partial$.

\begin{proof}
We give only a sketch of the proof, because it is rather standard. Fundamental matrix for \eqref{system for eta} satisfies the integral equation
\begin{equation}\label{equation for Phi}
    \Phi(x_0,x,\lambda)=I+\int_{x_0}^xL(t,\lambda)\Phi(x_0,t,\lambda)dt.
\end{equation}
To study $M_n(\lambda)$ put $x_0=a(n-1)$ and consider Volterra integral operator $\mathcal V_n(\lambda)$ in Banach space $C([0;a],M^{2,2})$ (of $2\times2$ matrix functions, with any matrix
norm) defined by the rule
\begin{equation*}
    \mathcal V_n(\lambda):u(x)\mapsto\int_0^xL(a(n-1)+t,\lambda)u(t)dt.
\end{equation*}
Using Neumann series and direct estimate of the norm we can show the existence of $(I-\mathcal V_n)^{-1}$ and the following estimate for its norm:
\begin{equation}\label{estimate of (I-V)^(-1)}
    \|(I-\mathcal V_n(\lambda))^{-1}\|\le\exp\l(\int_{a(n-1)}^{an}\|L(t,\lambda)\|\, dt\r).
\end{equation}
Returning to \eqref{equation for Phi} we can write the matrix equality
\begin{multline*}
    \Phi(a(n-1),a(n-1)+x,\lambda)=((I-\mathcal V_n(\lambda))^{-1}I)(x)=
    \\=I+(\mathcal V_n(\lambda)I)(x)+(\mathcal V_n^2(\lambda)(I-\mathcal V_n(\lambda))^{-1}I)(x),
\end{multline*}
and so putting $x=a$ we have:
\begin{equation*}
    M_n(\lambda)=I+\int_{a(n-1)}^{an}L(t,\lambda)dt+
    \underbrace{(\mathcal V_n^2(\lambda)(I-\mathcal V_n(\lambda))^{-1}I)(a)}_{=:R_n^{(2)}(\lambda)}.
\end{equation*}
The remainder $R^{(2)}_n(\cdot)$ is continuous in $\sigma(\mathcal L_{per})\backslash\partial$ for every $n$ and satisfies the uniform estimate
\begin{equation*}
    R_n^{(2)}(\lambda)=O\left(\l(\int_{a(n-1)}^{an}\|L(t,\lambda)\|\,dt\r)^2\right)\text{ as }n\rightarrow\infty.
\end{equation*}
Fix the compact set $K\subset\sigma(\mathcal L_{per})\backslash\partial$. One can always choose Bloch solutions so that their Wronskian is analytic functions and does not vanish on $\sigma(\mathcal L_{per})$. From the properties of Bloch solutions it follows that there exists $c_1(K)$ such that for every $\lambda\in K$ and $x$,
\begin{equation*}
    \l\|
    \frac1{W\{\psi_+(\lambda),\psi_-(\lambda)\}}
    \left(%
    \begin{array}{l}
    -\psi_+(x,\lambda)\psi_-(x,\lambda)\ \  -\psi_+^2(x,\lambda) \\
    \ \ \psi_-^2(x,\lambda)\ \ \ \ \ \ \psi_+(x,\lambda)\psi_-(x,\lambda) \\
    \end{array}%
    \right)
    \r\|
    \le c_1(K).
\end{equation*}
A straightforward estimate yields:
\begin{equation*}
    \int_{a(n-1)}^{an}\|L(t,\lambda)\|\,dt
    \le
    c_1(K)\left(|c|\ln\l(\frac{an+1}{an+1-a}\r)+\int_{a(n-1)}^{an}|q_1(t)|dt\right),
\end{equation*}
and thus $\{R_n^{(2)}(\lambda)\}_{n=1}^{\infty}\in l^1(K)$. Now consider the expression for
$\int_{a(n-1)}^{an}L(t,\lambda)dt$. Fix the parameter $t$ in the denominator, putting it equal to $an$ (this gives precisely the expression for the integral in the assertion of the lemma). The difference is:
\begin{equation}\label{second error term}
    \int_{a(n-1)}^{an}
    \frac{c\sin(2\omega t+\delta)}{W\{\psi_+(\lambda),\psi_-(\lambda)\}}
    \l(\frac1{t+1}-\frac1{an}\r)
    \left(%
    \begin{array}{cc}
    -\psi_+(t,\lambda)\psi_-(t,\lambda) & -\psi_+^2(t,\lambda) \\
    \psi_-^2(t,\lambda) & \psi_+(t,\lambda)\psi_-(t,\lambda) \\
    \end{array}%
    \right)dt,
\end{equation}
which admits an estimate by $\frac{c_1(K)|c|(a+1)}{n(an+1-a)}$. The total error term $R^{(1)}_n$ is the sum of $R^{(2)}_n$ and the expression \eqref{second error term}, and therefore it possesses the required properties. This completes the proof.
\end{proof}

\subsection*{Step 2. Fourier decomposition of Bloch solutions}
Rewrite the expression from the second term of the formula for $M_n(\lambda)$ given by Lemma \ref{lem M-n(lambda)},
\begin{equation}\label{integral - oscillating term of M-n}
    \frac1a\int_{a(n-1)}^{an}
    \frac{c\sin(2\omega t+\delta)}{W\{\psi_+(\lambda),\psi_-(\lambda)\}}
    \left(%
    \begin{array}{cc}
    -\psi_+(t,\lambda)\psi_-(t,\lambda) & -\psi_+^2(t,\lambda) \\
    \psi_-^2(t,\lambda) & \psi_+(t,\lambda)\psi_-(t,\lambda) \\
    \end{array}%
    \right)dt,
\end{equation}
using Fourier decompositions for periodic parts of Bloch solutions. Let $\{b_l^+(\lambda)\}_{l=-\infty}^{\infty}$, $\{b_l^-(\lambda)\}_{l=-\infty}^{\infty}$ and $\{b_l(\lambda)\}_{l=-\infty}^{\infty}$ be Fourier coefficients defined by identities
\begin{equation}\label{Fourier decomposition of Bloch solutions}
    \psi_+(x,\lambda)\psi_-(x,\lambda)\equiv\sum\limits_{l=-\infty}^{+\infty}b_l(\lambda)e^{2i\pi l\frac xa},
    \
    \psi_{\pm}^2(x,\lambda)\equiv\sum\limits_{l=-\infty}^{+\infty}b^{\pm}_l(\lambda)e^{2i(\pi l\pm k(\lambda))\frac xa}.
\end{equation}
These Fourier coefficients and their derivatives with respect to $\lambda$ are locally uniformly bounded in $\sigma(\mathcal
L_{per})\backslash\partial$ and obey the locally uniform (on the same set) estimate $O(1/l^2)$ as $l\rightarrow\infty$. This fact is rather standard, see, e.g., \cite{Kurasov-Simonov-2011} for the details. Since $\psi_+$ and $\psi_-$ are complex conjugate on $\sigma(\mathcal
L_{per})\backslash\partial$, one has:
\begin{equation*}
    b^+_l(\lambda)=\overline{b^-_{-l}(\lambda)},
    \
    b_l(\lambda)=\overline{b_{-l}(\lambda)}\text{ for every }l\in\mathbb Z\text{ and }\lambda\in\sigma(\mathcal
    L_{per})\backslash\partial.
\end{equation*}
Substituting identities \eqref{Fourier decomposition of Bloch solutions} into the expression \eqref{integral - oscillating term of M-n} and changing the order of summation and integration (which is possible due to the properties of Fourier coefficients mentioned above), one obtains the following result.
Expression \eqref{integral - oscillating term of M-n} equals to
\linebreak
$\left(%
\begin{array}{cc}
\beta_n^{(d)}(\lambda) & \beta_n^{(ad)}(\lambda) \\
\overline{\beta_n^{(ad)}}(\lambda) & -\beta_n^{(d)}(\lambda) \\
\end{array}%
\right)$,
and the monodromy matrix can be written in the form
\begin{equation}\label{M-n final}
    M_n(\lambda)=I+\frac1n
    \left(%
    \begin{array}{cc}
    \beta_n^{(d)}(\lambda) & \beta_n^{(ad)}(\lambda) \\
    \overline{\beta_n^{(ad)}}(\lambda) & -\beta_n^{(d)}(\lambda) \\
    \end{array}%
    \right)
    +R_n^{(1)}(\lambda).
\end{equation}
Here we have introduced the following notations:
\begin{equation}\label{beta-0,+,-}
    \begin{array}{l}
    \beta_0(\lambda):=-\frac{ce^{i(\delta-a\omega)}}{2iW\{\psi_+(\lambda),\psi_-(\lambda)\}}\sum\limits_{l=-\infty}^{+\infty}b_l(\lambda)
    \frac{\sin(a\omega)}{\pi l+a\omega},
    \\
    \beta_{\pm}(\lambda):=\mp\frac{ce^{i(\delta-(k(\lambda)\pm a\omega))}}{2iW\{\psi_+(\lambda),\psi_-(\lambda)\}}
    \sum\limits_{l=-\infty}^{+\infty}b^+_l(\lambda)\frac{\sin(k(\lambda)\pm a\omega)}{\pi l+k(\lambda)\pm a\omega},
    \end{array}
\end{equation}
and
\begin{equation}\label{beta-d,ad}
    \begin{array}{l}
    \beta_n^{(d)}(\lambda):=\beta_0(\lambda)e^{2ia\omega n}-\overline{\beta_0(\lambda)}e^{-2ia\omega n},
    \\
    \beta_n^{(ad)}(\lambda):=\beta_+(\lambda)e^{2i(k(\lambda)+a\omega)n}+\beta_-(\lambda)e^{2i(k(\lambda)-a\omega)n}.
    \end{array}
\end{equation}

\subsection*{Step 3. Elimination of non-resonant terms by Harris-Lutz method}
For every fixed $\lambda\in\sigma(\mathcal L_{per})\backslash\partial$ the entries of the matrix
$\left(%
\begin{array}{cc}
\beta_n^{(d)}(\lambda) & \beta_n^{(ad)}(\lambda) \\
\overline{\beta_n^{(ad)}}(\lambda) & -\beta_n^{(d)}(\lambda) \\
\end{array}%
\right)$
are linear combinations of four exponential terms: $e^{2ia\omega n}$, $e^{-2ia\omega n}$, $e^{2i(k(\lambda)+a\omega)n}$ and $e^{2i(k(\lambda)-a\omega)n}$. The first two of them do not depend on $\lambda$ and since $\omega\notin\frac{\pi\mathbb Z}{2a}$ they do oscillate. The third and fourth terms do depend on $\lambda$. They oscillate, if $k(\lambda)+a\omega\notin\pi\mathbb Z$ and $k(\lambda)-a\omega\notin\pi\mathbb Z$, respectively, otherwise they are constant (we call resonance this matching of the quasi-momentum $k(\lambda)$ and the frequency $\omega$). Oscillating exponential terms can be dropped with the help of Harris-Lutz transformation and do not affect the asymptotical type of solutions of the system $w_{n+1}=M_n(\lambda)w_n$. Therefore (see the rigorous proof of this fact below) if $\lambda$ is not one of the resonance points $\nu_{j,\pm}$, which are defined by conditions
\begin{equation*}
    k(\nu_{j,+}):=\pi\left(j+1-\left\{\frac{a\omega}{\pi}\right\}\right),
    \
    k(\nu_{j,-}):=\pi\left(j+\left\{\frac{a\omega}{\pi}\right\}\right),
    \
    j\ge0,
\end{equation*}
then every solution of the system $w_{n+1}=M_n(\lambda)w_n$ has a limit as $n\rightarrow\infty$. Note that due to the condition $\omega\notin\frac{\pi\mathbb Z}{2a}$, for every $j\ge0$ resonance points $\nu_{j,+}\neq\nu_{j,-}$ do not coincide with the endpoints of the $j$-th spectral band.

Pick a spectral band and one of the two critical points in this band, e.g., $\nu_{j,+}$ for some index $j$. To simplify the notation, we write $\nu_{cr}:=\nu_{j,+}$. Consider an open neighborhood $U_{cr}$ of the point $\nu_{cr}$ such that it contains neither the critical point $\nu_{j,-}$ nor the endpoints of the band $\mu_j$ and $\lambda_j$. In what follows we assume that $\nu_{cr}$ and $U_{cr}$ are fixed and we drop the indices $j,+$ in most cases. To remove the oscillating terms from the system we need the following elementary lemma belonging to the domain of the mathematical folklore.

    \begin{lem}\label{lem estimate of sum for Harris-Lutz}
    For every real $\xi\notin2\pi\mathbb Z$ and $n\in\mathbb N$,
    \begin{equation*}
        \l|\sum_{m=n}^{\infty}\frac{e^{im\xi}}m\r|\le\frac1{n\l|\sin\frac{\xi}2\r|}.
    \end{equation*}
    \end{lem}

\begin{proof}
\begin{multline*}
    \left|(e^{i\xi}-1)\sum_{m=n}^{\infty}\frac{e^{im\xi}}m\right|
    =\left|\sum_{m=n}^{\infty}e^{i(m+1)\xi}\l(\frac1m-\frac1{m+1}\r)-\frac{e^{in\xi}}n\right|
    \\
    \le\sum_{m=n}^{\infty}\l(\frac1m-\frac1{m+1}\r)+\frac1n=\frac2n.
\end{multline*}
Since $|e^{i\xi}-1|=2\l|\sin\frac{\xi}2\r|$, this argument completes the proof.
\end{proof}

By a Harris-Lutz transformation uniform in $U_{cr}$ it is possible to eliminate non-res-
\linebreak
onating exponential terms and to stabilize coefficients at resonating terms (i.e., make them independent of $\lambda$). Doing this we can provide for the uniform summability of the remainder. We formulate this argument in the following lemma.

    \begin{lem}\label{lem properties of Harris-Lutz transformation}
    There exists a sequence of matrices $\{Q^{(1)}_n(\lambda)\}_{n=1}^{\infty}$ defined in $U_{cr}$ such that $Q^{(1)}_n(\lambda)=O\l(\frac1n\r)$ as $n\rightarrow\infty$ uniformly in $U_{cr}$ and
    \begin{multline}\label{uniform Harris-Lutz}
        \exp\Bigl(-Q^{(1)}_{n+1}(\lambda)\Bigr)M_n(\lambda)\exp\Bigl(Q^{(1)}_n(\lambda)\Bigr)
        \\
        =I+\frac1n
        \left(%
        \begin{array}{cc}
        0 & \beta_+(\nu_{cr})e^{2i(k(\lambda)+a\omega)n} \\
        \overline{\beta_+(\nu_{cr})}e^{-2i(k(\lambda)+a\omega)n} & 0 \\
        \end{array}%
        \right)
        +R^{(3)}_n(\lambda)
    \end{multline}
    with some $\{R^{(3)}_n(\lambda)\}_{n=1}^{\infty}\in l^1(U_{cr})$ such that $R^{(3)}_n(\cdot)$ is continuous in $U_{cr}$ for every $n$.
    \end{lem}

\begin{proof}
We follow the scheme of \cite{Benzaid-Lutz-1987} and need to take care of the uniformity only.  The explicit formula for $Q^{(1)}_n$ is
\begin{equation}\label{Q-1}
    Q^{(1)}_n(\lambda):=
    -\sum_{m=n}^{\infty}\frac1m
    \left(
      \begin{array}{c}
        \quad \beta_m^{(d)}(\lambda)\quad\quad  \beta_m^{(ad)}(\lambda)-\beta_+(\nu_{cr})e^{2i(k(\lambda)+a\omega)m} \\
        \overline{\beta_m^{(ad)}(\lambda)}-\overline{\beta_+(\nu_{cr})}e^{-2i(k(\lambda)+a\omega)m}\quad\quad -\beta_m^{(d)}(\lambda) \\
      \end{array}
    \right).
\end{equation}
Lemma \ref{lem estimate of sum for Harris-Lutz} yields immediately the estimates
\begin{equation*}
    \l|\sum_{m=n}^{\infty}\frac{\beta_m^{(d)}(\lambda)}m\r|\le\frac{2|\beta_0(\lambda)|}{n|\sin a\omega|}
\end{equation*}
and
\begin{equation*}
    \l|\sum_{m=n}^{\infty}\frac{\beta_m^{(ad)}(\lambda)-\beta_+(\nu_{cr})e^{2i(k(\lambda)+a\omega)}}m\r|
    \le
    \frac{|\beta_+(\lambda)-\beta_+(\nu_{cr})|}{n|\sin(k(\lambda)+a\omega)|}
    +
    \frac{|\beta_-(\lambda)|}{n|\sin(k(\lambda)-a\omega)|}.
\end{equation*}
Functions $\beta_0$ and $\beta_-$ are continuous and the denominators $\sin a\omega$ and $\sin(k(\lambda)-a\omega)$ are separated from zero in $U_{cr}$. On the other hand, the denominator $\sin(k(\lambda)+a\omega)$ has the only zero at the point $\nu_{cr}$, which is simple and thus is compensated by the zero of the numerator. Therefore the estimate $Q^{(1)}_n(\lambda)=O\l(\frac1n\r)$ holds and is uniform in $U_{cr}$. Using the obvious property
\begin{equation*}
    Q^{(1)}_{n+1}(\lambda)-Q^{(1)}_n(\lambda)=\frac1n
    \left(%
    \begin{array}{l}
      \ \ \ \ \ \ \ \ \
      \beta_n^{(d)}(\lambda)
      \ \ \ \ \ \ \ \ \
      \beta_n^{(ad)}(\lambda)-\beta_+(\nu_{cr})e^{2i(k(\lambda)+a\omega)n}
      \\
      \overline{\beta_n^{(ad)}(\lambda)}-\overline{\beta_+(\nu_{cr})}e^{-2i(k(\lambda)+a\omega)n}
      \ \ \ \ \ \ \ \
      -\beta_n^{(d)}(\lambda) \\
    \end{array}%
    \right),
\end{equation*}
one obtains:
\begin{multline*}
    M_n(\lambda)=I+\frac1n
        \left(%
    \begin{array}{cc}
    0 & \beta_+(\nu_{cr})e^{2i(k(\lambda)+a\omega)n} \\
    \overline{\beta_+(\nu_{cr})}e^{-2i(k(\lambda)+a\omega)n} & 0 \\
    \end{array}%
    \right)
    \\
    +Q^{(1)}_{n+1}(\lambda)-Q^{(1)}_n(\lambda)+R^{(1)}_n(\lambda).
\end{multline*}
Multiplying by $\exp(Q^{(1)}_n(\lambda))$ from the right; by $\exp(-Q^{(1)}_{n+1}(\lambda))$ form the left, expanding exponents and absorbing the terms of the order $1/n^2$ in the remainder, we have:
\begin{multline*}
    \exp\Bigl(-Q^{(1)}_{n+1}(\lambda)\Bigr)M_n(\lambda)\exp\Bigl(Q^{(1)}_n(\lambda)\Bigr)
    \\
    =I+\frac1n
    \left(%
    \begin{array}{cc}
    0 & \beta_+(\nu_{cr})e^{2i(k(\lambda)+a\omega)n} \\
    \overline{\beta_+(\nu_{cr})}e^{-2i(k(\lambda)+a\omega)n} & 0 \\
    \end{array}%
    \right)
    +\underbrace{R_n^{(1)}(\lambda)+O\l(\frac1{n^2}\r)}_{=R^{(3)}_n(\lambda)},
\end{multline*}
where the estimate $O\l(\frac1{n^2}\r)$ is uniform in $U_{cr}$. It is clear that $R^{(3)}_n(\lambda)$ is
continuous in $U_{cr}$ for every $n$. This completes the proof.
\end{proof}

\subsection*{Step 4. Reduction to the model problem}
After the Harris-Lutz transformation $\{w_n\}_{n=1}^{\infty}\mapsto\{\widetilde w_n\}_{n=1}^{\infty}$ with
\linebreak
$\widetilde w_n:=\exp(-Q^{(1)}_n(\lambda))w_n$ the system $w_{n+1}=M_n(\lambda)w_n$ is reduced to the system
\begin{equation*}
    \widetilde w_{n+1}=
    \left[
    I+\frac1n
    \left(%
    \begin{array}{cc}
    0 & \beta_+(\nu_{cr})e^{2i(k(\lambda)+a\omega)n} \\
    \overline{\beta_+(\nu_{cr})}e^{-2i(k(\lambda)+a\omega)n} & 0 \\
    \end{array}%
    \right)
    +R^{(3)}_n(\lambda)
    \right]
    \widetilde w_n.
\end{equation*}
At the critical point it takes the following form:
\begin{equation*}
    \widetilde w_{n+1}=
    \left[
    I+\frac1n
    \left(%
    \begin{array}{cc}
    0 & \beta_+(\nu_{cr})\\
    \overline{\beta_+(\nu_{cr})} & 0 \\
    \end{array}%
    \right)
    +R_n^{(3)}(\nu_{cr})
    \right]
    \widetilde w_n.
\end{equation*}
It is convenient to diagonalize the constant matrix in the second term. One has:
\begin{multline*}
    \left(%
        \begin{array}{cc}
        e^{\frac i2\arg \beta_+(\nu_{cr})} & ie^{\frac i2\arg \beta_+(\nu_{cr})} \\
        e^{-\frac i2\arg \beta_+(\nu_{cr})} & -ie^{-\frac i2\arg \beta_+(\nu_{cr})} \\
        \end{array}%
    \right)^{-1}
    \left(%
    \begin{array}{cc}
    0 & \beta_+(\nu_{cr})\\
    \overline{\beta_+(\nu_{cr})} & 0 \\
    \end{array}%
    \right)
    \\
    \times
    \left(%
        \begin{array}{cc}
        e^{\frac i2\arg \beta_+(\nu_{cr})} & ie^{\frac i2\arg \beta_+(\nu_{cr})} \\
        e^{-\frac i2\arg \beta_+(\nu_{cr})} & -ie^{-\frac i2\arg \beta_+(\nu_{cr})} \\
        \end{array}%
    \right)
    =|\beta_+(\nu_{cr})|
    \left(
      \begin{array}{cc}
        1 & 0 \\
        0 & -1 \\
      \end{array}
    \right).
\end{multline*}
Consider the following sequence $\{v_n\}_{n=1}^{\infty}$ instead of $\{\widetilde w_n\}_{n=1}^{\infty}$:
\begin{equation*}
    v_n:=
    \left(%
        \begin{array}{cc}
        e^{\frac i2\arg \beta_+(\nu_{cr})} & ie^{\frac i2\arg \beta_+(\nu_{cr})} \\
        e^{-\frac i2\arg \beta_+(\nu_{cr})} & -ie^{-\frac i2\arg \beta_+(\nu_{cr})} \\
        \end{array}%
    \right)
    ^{-1}
    \widetilde w_n.
\end{equation*}
Due to \eqref{uniform Harris-Lutz} the system $w_{n+1}=M_n(\lambda)w_n$ is equivalent to:
\begin{equation}\label{system for v}
    v_{n+1}=
    \l[
    I+\frac{\beta}n\left(%
    \begin{array}{cc}
    \cos(2(k(\lambda)+a\omega)n) & \sin(2(k(\lambda)+a\omega)n) \\
    \sin(2(k(\lambda)+a\omega)n) & -\cos(2(k(\lambda)+a\omega)n) \\
    \end{array}%
    \right)
    +R^{(4)}_n(\lambda)
    \r]
    v_n,
\end{equation}
where the sequence $\{R^{(4)}_n(\lambda)\}_{n=1}^{\infty}\in l^1(U_{cr})$, the functions $R^{(4)}_n(\cdot)$ are continuous in $U_{cr}$ and
\begin{equation*}
    \beta:=|\beta_+(\nu_{cr})|
    =\frac{|c|}{2a|W\{\psi_+(\nu_{cr}),\psi_-(\nu_{cr})\}|}
    \left|
    \int_0^{a}\psi_+^2(t,\nu_{cr})e^{2i\omega t}dt
    \right|.
\end{equation*}
Replace the parameter $\lambda$ on the set $U_{cr}$ by the new small parameter
\begin{equation}\label{epsilon}
    \varepsilon:=2(k(\lambda)-k(\nu_{cr}))=2(k(\lambda)+a\omega)-2\pi\left(j+1+\left\lfloor\frac{a\omega}{\pi}\right\rfloor\right),
\end{equation}
where $\lfloor\cdot\rfloor$ is the standard floor function. The set of values taken by $\varepsilon$ is $U:=\{2(k(\lambda)-k(\nu_{cr})),\lambda\in U_{cr}\}$. By the property of the quasi-momentum (that $k(\lambda)\in[\pi j;\pi(j+1)]$ in $j$-th spectral band) and the condition $\omega\notin\frac{\pi\mathbb Z}{2a}$, which guarantees that the critical point is in the interior of the spectral band, we have $\overline U\subset (-2\pi;2\pi)$. Denote $R_n(\varepsilon)=R^{(4)}_n(\lambda)$ for $\lambda$ corresponding to $\varepsilon$ according to \eqref{epsilon}. System \eqref{system for v} then reads:
\begin{equation}\label{system for v with epsilon}
    v_{n+1}=
    \l[
    I+\frac{\beta}n\left(%
    \begin{array}{cc}
    \cos(\varepsilon n) & \sin(\varepsilon n) \\
    \sin(\varepsilon n) & -\cos(\varepsilon n) \\
    \end{array}%
    \right)
    +R_n(\varepsilon)
    \r]
    v_n,
\end{equation}

The aim of this section is to rewrite the Weyl-Titchmarsh type formula in terms of the solutions of the system \eqref{system for v with epsilon}. Proposition \ref{prop Weyl-Titchmarsh type formula} deals with the solution $\varphi_{\alpha}$ of the spectral equation for the operator $\mathcal L_{\alpha}$. Combining all the transformations described above and denoting the result by $\Xi:\psi(x)\mapsto\{v_n\}_{n=1}^{\infty}$, one ends up with the following model image of the initial solution $\varphi_{\alpha}$:
\begin{multline}\label{v-alpha}
    v_{\alpha,n}(\varepsilon)=v_{\alpha,n}(2(k(\lambda)-k(\nu_{cr}))):=
    \left(%
        \begin{array}{cc}
        e^{\frac i2\arg \beta_+(\nu_{cr})} & ie^{\frac i2\arg \beta_+(\nu_{cr})} \\
        e^{-\frac i2\arg \beta_+(\nu_{cr})} & -ie^{-\frac i2\arg \beta_+(\nu_{cr})} \\
        \end{array}%
    \right)
    ^{-1}
    \\
    \times
    \exp\bigl(-Q^{(1)}_n(\lambda)\bigr)
    \left(%
    \begin{array}{cc}
    \psi_-(a(n-1),\lambda) & \psi_+(a(n-1),\lambda) \\
    \psi_-'(a(n-1),\lambda) & \psi_+'(a(n-1),\lambda) \\
    \end{array}%
    \right)^{-1}
    \left(%
    \begin{array}{c}
    \varphi_{\alpha}(a(n-1),\lambda) \\
    \varphi_{\alpha}'(a(n-1),\lambda) \\
    \end{array}%
    \right).
\end{multline}
This is obviously a solution of the system \eqref{system for v with epsilon} and it is continuous in $U$ for every $n$.

    \begin{lem}\label{lem Weyl-Titchmarsh formula}
    For every $\varepsilon\in U\backslash\{0\}$ there exists a limit $\lim\limits_{\ninf}v_{\alpha,n}(\varepsilon)\neq0$, which is continuous in $\varepsilon\in U\backslash\{0\}$ as a vector-valued function of $\varepsilon$. The spectral density of $\mathcal L_{\alpha}$ equals
    \begin{equation}\label{Weyl-Titchmarsh type formula v}
        \rho'_{\alpha}(\lambda)
        =\frac2{\pi|W\{\psi_+(\lambda),\psi_-(\lambda)\}|\l\|\lim\limits_{\ninf}v_{\alpha,n}(2(k(\lambda)-k(\nu_{cr})))\r\|^2}
        \text{ a.e. in }U_{cr}.
    \end{equation}
    \end{lem}

\begin{proof}
A straightforward substitution of the transformation $\Xi$ yields this result.
Indeed, asymptotics of the solution $\varphi$ and of its derivative \eqref{asymptotics of phi-alpha} together with the boundedness of
$\left(%
\begin{array}{cc}
\psi_-(x,\lambda) & \psi_+(x,\lambda) \\
\psi_-'(x,\lambda) & \psi_+'(x,\lambda) \\
\end{array}%
\right)^{-1}$
imply that there exists the limit
\begin{equation*}
    \lim_{x\rightarrow+\infty}
    \left(%
    \begin{array}{cc}
    \psi_-(x,\lambda) & \psi_+(x,\lambda) \\
    \psi_-'(x,\lambda) & \psi_+'(x,\lambda) \\
    \end{array}%
    \right)^{-1}
    \left(%
    \begin{array}{c}
    \varphi_{\alpha}(x,\lambda) \\
    \varphi_{\alpha}'(x,\lambda) \\
    \end{array}%
    \right)
    =
    \left(%
    \begin{array}{c}
    A_{\alpha}(\lambda) \\
    \overline{A_{\alpha}}(\lambda) \\
    \end{array}%
    \right).
\end{equation*}
Furthermore, since $\exp\bigl(-Q^{(1)}_n(\lambda)\bigr)\rightarrow I$ as $\ninf$, there exists the limit
\begin{equation*}
    \lim_{\ninf}v_{\alpha,n}(\varepsilon)=\frac12
    \left(%
    \begin{array}{cc}
    1 & 1 \\
    -i & i \\
    \end{array}%
    \right)
    \left(%
    \begin{array}{c}
    e^{-\frac i2\arg \beta_+(\nu_{cr})}A_{\alpha}(\lambda) \\
    e^{\frac i2\arg \beta_+(\nu_{cr})}\overline{A_{\alpha}}(\lambda) \\
    \end{array}%
    \right).
\end{equation*}
It follows that $|A_{\alpha}(\lambda)|^2=\l\|\lim\limits_{\ninf}v_{\alpha,n}(\varepsilon)\r\|^2/4$, and substitution of this to the Weyl-\linebreak Titchmarsh type formula \eqref{Weyl-Titchmatsh type formula phi} completes the proof.
\end{proof}

The summary of this section is that the study of the spectral density of $\mathcal L_{\alpha}$ is reduceable to the study of the system \eqref{system for v with epsilon} and, more precisely, to the study of the behavior of $\lim\limits_{\ninf}v_{\alpha,n}(\varepsilon)$ for small $\varepsilon$.

In the general case ($\nu_{cr}=\nu_{j,\pm}$) denote the coefficients $\beta$, which may be different at different resonance points, as $\beta_{j,\pm}$. Explicit calculations give: $\beta_{j,\pm}=\left|\frac{c\int_0^{a}\psi_{\pm}^2(t,\nu_{j,\pm})e^{2i\omega t}dt}{2aW\{\psi_+(\nu_{j,\pm}),\psi_-(\nu_{j,\pm})\}}\right|$.

\begin{rem}
Coefficients $\beta_{j,\pm}$ are not necessarily non-zero, because they are proportional to the Fourier coefficients of $p_+(\cdot,\nu_{j,\pm})$, which might be zero. E.g., consider the case of zero periodic potential, $q(x)\equiv0$. In this case one can choose the period $a$ arbitrarily and the result is independent of the choice (except for the case $\omega\in\frac{\pi\mathbb Z}a$). For any fixed $a$, the half-line is divided into spectral bands with coinciding endpoints,
$\left[\bigl(\frac{\pi j}a\bigr)^2;\bigl(\frac{\pi(j+1)}a\bigr)^2\right]$, $j\ge0$.
The quasi-momentum is $k(\lambda)=a\sqrt\lambda$, the critical points are $\nu_{j,+}=\left(\frac{\pi}a\left(j+1-\left\{\frac{a\omega}{\pi}\right\}\right)\right)^2$
and
$\nu_{j,-}=\left(\frac{\pi}a\left(j+\left\{\frac{a\omega}{\pi}\right\}\right)\right)^2$, $j\ge0$.
Bloch solutions are $\psi_{\pm}(x,\lambda)=e^{\pm ik(\lambda)\frac xa}=e^{\pm\sqrt\lambda x}$ and their periodic parts $p_+(x,\lambda)\equiv p_-(x,\lambda)\equiv1$ have only one non-zero Fourier coefficient. Explicit calculation shows that $\beta_{j,+}\equiv0$ and $\beta_{j,-}=0$ for every $j\neq\lfloor\frac{a\omega}{\pi}\rfloor$. For the single existing resonance point one has: $\nu_{\lfloor\frac{a\omega}{\pi}\rfloor,-}=\omega^2$ and $\beta_{\lfloor\frac{a\omega}{\pi}\rfloor,-}=\frac{|c|}{4\omega}$. This coincides with the classical results on the Wigner-von Neumann potential \cite{Wigner-von-Neumann-1929}. Our result concerning zeroes of the spectral density in this case is in accordance with the result of Hinton-Klaus-Shaw \cite{Hinton-Klaus-Shaw-1991}.
\end{rem}

\section{Model problem}\label{section model problem}
In this section our aim is to study the dependence on $\varepsilon$, which can be arbitrarily small, of the limits of solutions to the system \eqref{system for v with epsilon}. As we have shown in the previous section, this is equivalent to the study of the behavior of the spectral density of $\mathcal L_{\alpha}$ near critical points.

Let us make few comments on the structure of the coefficient matrix of the model system. One can write it as follows:
$I+\frac{\beta}nD_{-\varepsilon n}\sigma_3+R_n(\varepsilon)$,
where $D_{\varphi}=
\left(
  \begin{array}{cc}
    \cos\varphi & \sin\varphi \\
    -\sin\varphi & \cos\varphi \\
  \end{array}
\right)$ is the matrix of rotation by the angle $\varphi$ and
$\sigma_3=
\left(
  \begin{array}{cc}
    1 & 0 \\
    0 & -1 \\
  \end{array}
\right)$ (notation for the Pauli matrix) is the reflection matrix. The presence of the latter plays a very important role. With $\sigma_3$, the system is elliptic for $\varepsilon\in U\backslash\{0\}$ and hyperbolic for $\varepsilon=0$ (we call the system elliptic, if it has a base of solutions of the same order of magnitude and hyperbolic in the opposite case). Without $\sigma_3$, there is no change of type of the system at the point $\varepsilon=0$, the system is always elliptic. If $\sigma_3$ is absent, one can factor out the diagonal term of the first two summands $1+\beta\frac{\cos(\varepsilon n)}n$ and the coefficient matrix reduces to
$I+\frac{\beta\sin(\varepsilon n)}n
\left(
\begin{array}{cc}
0 & 1 \\
1 & 0 \\
\end{array}
\right)
+R_n(\varepsilon)$ with some other uniformly summable sequence $\{R_n(\varepsilon)\}_{n=1}^{\infty}$. The matrix in the second term here is constant and diagonalizable. Its spectrum is purely imaginary, and the Levinson theorem immediately gives the answer (=asymptotics of solutions), which is uniform in $\varepsilon\in U$ including the point $\varepsilon=0$. We could say that the problem is "scalarized" in this case. The situation is different in our case, since the difference between the two eigenvalues is not pure imaginary: $\sigma\left(I+\frac{\beta}nD_{-\varepsilon n}\sigma_3\right)=\left\{1\pm\frac{\beta}n\right\}$. This leads to serious troubles in the analysis and exhibits a new phenomenon.

One may expect that for sufficiently small values of $\varepsilon$ the magnitude of solutions is determined mostly by diagonal elements $1\pm\beta\frac{\cos(\varepsilon n)}n$. We want to transform the system in a way such that for every $\varepsilon\in U$ (including $\varepsilon=0$) the limit of every solution as $n\rightarrow\infty$ exists (for the system \eqref{system for v with epsilon} this is not true: if $\varepsilon=0$, then one of the solutions grows as $n^{\beta}$). To this end, we make the following substitution: $v_n=\exp\left(\beta\int_1^n\frac{\cos(\varepsilon r)}rdr\right)u_n$. This leads to the system $u_{n+1}=B_n(\varepsilon)u_n$ with
\begin{equation}\label{B-n}
    B_n(\varepsilon):=\exp\left(-\beta\int_n^{n+1}\frac{\cos(\varepsilon r)}rdr\right)
    \left[
    I+\frac{\beta}n
    \left(%
    \begin{array}{cc}
    \cos(\varepsilon n) & \sin(\varepsilon n) \\
    \sin(\varepsilon n) & -\cos(\varepsilon n) \\
    \end{array}%
    \right)
    +R_n(\varepsilon)
    \right].
\end{equation}
The existence of the limit of any solution of the system $u_{n+1}=B_n(\varepsilon)u_n$ is equivalent to the convergence of the infinite product $\prod_{n=1}^{\infty}B_n(\varepsilon)$ (in fact for $\varepsilon\neq0$ this follows from Lemma \ref{lem Weyl-Titchmarsh formula}). Moreover, most of the statements that we make about the asymptotic behavior of solutions of the system $u_{n+1}=B_n(\varepsilon)u_n$ can be formulated in terms of products of matrices $B_n(\varepsilon)$. In what follows we will choose the way of formulation depending on the convenience of its use. We are going to work with particular solutions determined by fixing their initial values, this is a discrete analogue of the Cauchy problem. To this end, we introduce the following notation: for given $\varepsilon\in U$ and the vector of initial data $f\in\mathbb C^2$ define the vector sequence $\{u_n(\varepsilon,f)\}_{n=1}^{\infty}$ by the recurrence relation
\begin{equation}\label{u(f)}
    \begin{array}{rl}
    u_1(\varepsilon,f) & :=f,
    \\
    u_{n+1}(\varepsilon,f) & :=B_n(\varepsilon)u_n(\varepsilon,f),\,n\ge1.
    \end{array}
\end{equation}

Note that due to the decomposition of the exponential term the matrix $B_n(\varepsilon)$ can be written as
\begin{equation*}
    B_n(\varepsilon)=I+\frac{\beta}n
    \left(
      \begin{array}{cc}
        \cos(\varepsilon n)-\int_n^{n+1}\cos(\varepsilon r)dr & \sin(\varepsilon n) \\
        \sin(\varepsilon n) & -\cos(\varepsilon n)-\int_n^{n+1}\cos(\varepsilon r)dr \\
      \end{array}
    \right)
    +\widetilde R_n(\varepsilon)
\end{equation*}
with a sequence $\{\widetilde R_n(\varepsilon)\}_{n=1}^{\infty}\in l^1(U)$. One can rewrite the system $u_{n+1}=B_n(\varepsilon)u_n$ as
\begin{equation}\label{system for u form with the difference}
    u_{n+1}-u_n=
    \left[
    \frac{\beta}n
    \left(
      \begin{array}{c}
        \cos(\varepsilon n)-\int_n^{n+1}\cos(\varepsilon r)dr \quad\quad\quad\quad \sin(\varepsilon n) \\
        \sin(\varepsilon n) \quad\quad\quad\quad -\cos(\varepsilon n)-\int_n^{n+1}\cos(\varepsilon r)dr \\
      \end{array}
    \right)
    +\widetilde R_n(\varepsilon)
    \right]
    u_n.
\end{equation}
The behavior of the solutions can be observed in a scale of the variable $y=n|\varepsilon|$ ("slow variable"). If one puts $z(y,\varepsilon,f):=u_{\left\lfloor\frac{y}{|\varepsilon|}\right\rfloor}(\varepsilon,f)$ and divides by $|\varepsilon|$, then \eqref{system for u form with the difference} becomes approximately
\begin{equation}\label{equation for z with difference}
    \frac{z(y+|\varepsilon|)-z(y)}{|\varepsilon|}\approx
    \left[
    \frac{\beta\,\text{sign}\,\varepsilon}y
    \left(
      \begin{array}{cc}
        0 & \sin y \\
        \sin y & -2\cos y \\
      \end{array}
    \right)
    +\frac1{|\varepsilon|}\widetilde R_{\left\lfloor\frac{y}{|\varepsilon|}\right\rfloor}(\varepsilon)
    \right]
    z(y)
\end{equation}
The remainder
$\frac1{|\varepsilon|}\widetilde R_{\left\lfloor\frac{y}{|\varepsilon|}\right\rfloor}(\varepsilon)$ is a step-wise constant matrix-valued function which is compressed in $\frac1{|\varepsilon|}$ times in the horizontal scale and stretched in $\frac1{|\varepsilon|}$ times in the vertical scale, therefore it concentrates near the origin of the variable $y$ and its $L_1$ norm is preserved. So we may expect that the remainder term will be absorbed into a new boundary condition of the limit problem. Expression on the left hand side of \eqref{equation for z with difference} becomes the derivative as $\varepsilon\rightarrow0$, and one "obtains" the following equation for the limits $h_{\pm}(y,f):=\lim\limits_{\varepsilon\rightarrow\pm0}z(y,\varepsilon,f)$:
\begin{equation}\label{equations for h-pm differential}
    h_{\pm}'(y)=\pm\frac{\beta}y
    \left(
      \begin{array}{cc}
        0 & \sin y \\
        \sin y & -2\cos y \\
      \end{array}
    \right)
    h_{\pm}(y),
\end{equation}
cf. \eqref{differential equation for h+ with no right hand side}. The remainder $\frac1{|\varepsilon|}\widetilde R_{\left\lfloor\frac{y}{|\varepsilon|}\right\rfloor}(\varepsilon)$ plays a role only in determining the initial values $h_{\pm}(0)$ for the solutions of \eqref{equations for h-pm differential}. It even suffices to know only $\widetilde R_n(0)$ to determine this initial value, under the condition of continuity of $\widetilde R_n(\varepsilon)$ for every $n$. Thus the picture of the whole phenomenon can be described as follows:
there exist two scales: "fast", discrete, $n\in\mathbb N$, and "slow", continuous, $y\in\mathbb R_+$. The system first moves along the first ("fast") scale with $\varepsilon=0$. The limit of the solution as $n\rightarrow\infty$ for $\varepsilon=0$ serves as initial value for the differential equation in the second ("slow") scale. Our aim in this section is to prove the following result which gives the exact formulation of the above considerations. We prefer to write an integral equation in the slow variable instead of the differential one, because the first has a unique solution while with the second one can have troubles due to the different behavior of the solutions near the origin in different cases depending on the value of $\beta$.

    \begin{thm}\label{thm discrete-continuous}
    Assume that functions $R_n(\cdot)$ are continuous in $U$ for every $n\in\mathbb N$, the matrices $B_n(\varepsilon)$ are invertible for every $n\in\mathbb N$, $\varepsilon\in U$ and the sequence $\{R_n(\varepsilon)\}_{n=1}^{\infty}\in l^1(U)$.
    Then for every $y>0$ and $f\in\mathbb C^2$ there exist two limits
    \begin{equation*}
        h_{\pm}(y,f):=\lim\limits_{\varepsilon\rightarrow\pm0}u_{\left\lfloor\frac y{|\varepsilon|}\right\rfloor}(\varepsilon,f),
    \end{equation*}
    which satisfy the following integral equations:
    \begin{equation}\label{equations for h}
    h_{\pm}(y,f)=\lim_{n\rightarrow\infty}u_n(0,f)
    \pm\int_0^y
    \left(
      \begin{array}{cc}
        0 & 1 \\
        \exp\left(-2\beta\int_t^y\frac{\cos s}sds\right) & 0 \\
      \end{array}
    \right)
    \frac{\beta\sin t}th_{\pm}(t,f)dt.
    \end{equation}
    Moreover, the following four limits exist and are equal:
    \begin{equation}\label{desired limit}
        \lim_{\varepsilon\rightarrow\pm0}\lim_{n\rightarrow\infty}u_n(\varepsilon,f)=\lim_{y\rightarrow+\infty}h_{\pm}(y,f).
    \end{equation}
    Additionally, the linear map $\Theta:f\mapsto \lim\limits_{n\rightarrow\infty}u_n(0,f)$ has rank one.
    \end{thm}

    \begin{rem}
    1. It follows that the linear map $f\mapsto\lim\limits_{\varepsilon\rightarrow\pm0}\lim\limits_{n\rightarrow\infty}u_n(\varepsilon,f)$ also has rank one.
    \\
    2. Note that $\lim\limits_{n\rightarrow\infty}u_n(0,f)
    =\lim\limits_{n\rightarrow\infty}\lim\limits_{\varepsilon\rightarrow\pm0}u_n(\varepsilon,f)$ and $\lim\limits_{\varepsilon\rightarrow\pm0}\lim\limits_{n\rightarrow\infty}u_n(\varepsilon,f)$ are the limits of the same expression taken in the different order and that they do not coincide. Our aim is to prove that the second limit of these two exists. It will follow then that the spectral density of $\mathcal L_{\alpha}$ can have zeros at critical points. In fact we can rewrite the second expression as
    $\lim\limits_{\varepsilon\rightarrow\pm0}\lim\limits_{y\rightarrow+\infty}u_{\left\lfloor\frac y{|\varepsilon|}\right\rfloor}(\varepsilon,f)$.
    Here the convergence of the first limit is uniform in $\varepsilon$ (unlike the convergence of $\lim\limits_{n\rightarrow\infty}u_n(\varepsilon,f)$ ), and this makes it possible to change the order of limits as in \eqref{desired limit}.
    \end{rem}

\begin{proof}
We divide the proof of Theorem \ref{thm discrete-continuous} into four steps.

\subsection*{Step I. A priori estimate and uniform convergence of the tail for the matrix product}
We start with few technical results concerning the system $u_{n+1}=B_n(\varepsilon)u_n$. These include uniform boundedness of its solutions (in both variables, $n$ and $\varepsilon$) and the uniform with respect to $\varepsilon$ convergence in the slow variable $y$.

    \begin{lem}\label{lem convergence of the product and degeneracy}
    Let $B_n(\varepsilon)$ be given by \eqref{B-n} and $\{R_n(\varepsilon)\}_{n=1}^{\infty}\in l^1(U)$. Then for every $\varepsilon\in U$ the product
    \begin{equation*}
        \prod_{n=1}^{\infty}B_n(\varepsilon)
    \end{equation*}
    converges. If matrices $B_n(\varepsilon)$ are invertible for every $n\in\mathbb N$ and $\varepsilon\in U$, then for every non-zero $\varepsilon$ the product is invertible, while for zero $\varepsilon$ it is of rank one.
    \end{lem}

\begin{proof}
This follows from the discrete Levinson theorem \cite{Benzaid-Lutz-1987}. Using the expansion
\begin{equation}\label{decomposition of the exponent}
    \exp\left(-\beta\int_n^{n+1}\frac{\cos r}rdr\right)=1-\frac{\beta}n\int_n^{n+1}\cos(\varepsilon r)dr+O\left(\frac1{n^2}\right)\text{ as }n\rightarrow\infty,
\end{equation}
we rewrite each matrix of the sequence $B_n(\varepsilon)$ in the following form (at places, we drop the argument $\varepsilon$ in order to simplify the notation, hoping that this will not lead to any confusion):
\begin{equation}\label{B-n form 2}
    B_n=I+V^{(1)}_n+R^{(5)}_n,
\end{equation}
where $\{R^{(5)}_n(\varepsilon)\}_{n=1}^{\infty}\in l^1(U)$ and
\begin{equation}\label{V-1}
    V^{(1)}_n(\varepsilon):=\frac{\beta}n
    \left(
    \begin{array}{cc}
    \cos(\varepsilon n)-\int_n^{n+1}\cos(\varepsilon r)dr & \sin(\varepsilon n) \\
    \sin(\varepsilon n) & -\cos(\varepsilon n)-\int_n^{n+1}\cos(\varepsilon r)dr \\
    \end{array}
    \right).
\end{equation}
One has:
\begin{equation}\label{equality for the integral}
    \int_n^{n+1}\cos(\varepsilon r)dr=\frac{\cos\varepsilon-1}{\varepsilon}\sin(\varepsilon n)+\frac{\sin\varepsilon}{\varepsilon}\cos(\varepsilon n).
\end{equation}

For $\varepsilon\neq0$ the conditions of Theorem 3.1 from \cite{Benzaid-Lutz-1987} are satisfied. The named theorem yields the existence of a base of solutions of the system $u_{n+1}=B_n(\varepsilon)u_n$, which have the asymptotics
\begin{equation*}
    \left(
      \begin{array}{c}
        1 \\
        0 \\
      \end{array}
    \right)
    +o(1)\text{ and }
    \left(
      \begin{array}{c}
        0 \\
        1 \\
      \end{array}
    \right)
    +o(1)\text{ as }n\rightarrow\infty.
\end{equation*}
This is equivalent to the convergence of the product $\prod_{n=1}^{\infty}B_n(\varepsilon)$ and its invertibility.

For $\varepsilon=0$ the matrix is reduced to
\begin{equation*}
    B_n=
    \left(
      \begin{array}{cc}
        1 & 0 \\
        0 & 1-\frac{2\beta}n \\
      \end{array}
    \right)
    +R^{(5)}_n,
\end{equation*}
and the base of solutions changes accordingly:
\begin{equation*}
    \left(
      \begin{array}{c}
        1 \\
        0 \\
      \end{array}
    \right)
    +o(1)\text{ and }
    \frac1{n^{2\beta}}
    \left[
    \left(
      \begin{array}{c}
        0 \\
        1 \\
      \end{array}
    \right)
    +o(1)
    \r]
    \text{ as }n\rightarrow\infty,
\end{equation*}
which follows from the discrete Levinson theorem \cite[Theorem 2.2]{Benzaid-Lutz-1987}. The existence of such a base of solutions is in its turn equivalent to the convergence of the product $\prod_{n=1}^{\infty}B_n(0)$ to a rank one matrix (since the second solution goes to zero as $n\rightarrow\infty$).
\end{proof}

    \begin{rem}\label{rem about Theta}
    The rank one matrix $\Theta=\prod_{n=1}^{\infty}B_n(0)$ defines the linear map \linebreak $f\mapsto\lim_{n\rightarrow\infty}u_n(0,f)$.
    \end{rem}

Asymptotics of solutions of the equation $u_{n+1}=B_n(\varepsilon)u_n$ given above
demonstrate the change of the system type. For non-zero $\varepsilon$ the system is elliptic (i.e., its solutions have the same rate of growth), while for $\varepsilon=0$ the system is hyperbolic (i.e., there exists a base of solutions, which have uncomparable magnitudes; this yields the existence of a subordinate solution).

The following lemma states the uniform convergence of the product of matrices $B_n(\varepsilon)$ in the slow scale.

    \begin{lem}\label{lem convergence of the tail}
    Let $\{R_n(\varepsilon)\}_{n=1}^{\infty}\in l^1(U)$. Then
    \begin{equation*}
        \prod_{n>\frac y{|\varepsilon|}}B_n(\varepsilon)\rightarrow I\text{ as }y\rightarrow+\infty\text{ uniformly in }\varepsilon\in U\backslash\{0\}.
    \end{equation*}
    \end{lem}

\begin{proof}
The sequence $V^{(1)}_n$ given by \eqref{V-1} has the following property:
\begin{equation}\label{estimate of the sum of V-1}
    \left\|\sum_{k\ge n}V^{(1)}_k(\varepsilon)\right\|
    \le
    \frac{4\beta}{n\left|\sin\frac{\varepsilon}2\right|}
    \text{ for every }n\in\mathbb N\text{ and }\varepsilon\in U\backslash\{0\}.
\end{equation}
This easily seen using the equality \eqref{equality for the integral}, elementary estimates $\left|\frac{\sin\varepsilon}{\varepsilon}\right|\le1$, $\left|\frac{\cos\varepsilon-1}{\varepsilon}\right|\le1$ and Lemma \ref{lem estimate of sum for Harris-Lutz}. It enables to define the sequence
\begin{equation*}
    Q^{(2)}_n:=\sum_{k=n}^{\infty}V^{(1)}_k.
\end{equation*}
Then
\begin{equation*}
    B_n=I+Q^{(2)}_n-Q^{(2)}_{n+1}+R^{(5)}_n.
\end{equation*}
Following the ideas of \cite{Harris-Lutz-1975,Benzaid-Lutz-1987}, we want to consider the Harris-Lutz transformation: $B^{(1)}_n:=\bigl(I-Q^{(2)}_{n+1}\bigr)^{-1}B_n\bigl(I-Q^{(2)}_n\bigr)$. If $n>\frac{c_2}{|\varepsilon|}$ with, say,
\begin{equation}\label{c-2}
    c_2:=8\beta\sup\limits_{\varepsilon\in U\backslash\{0\}}\frac{|\varepsilon|}{\left|\sin\frac{\varepsilon}2\right|},
\end{equation}
then it is easy to see that the estimate $\|Q^{(2)}_n(\varepsilon)\|<\frac12$ holds yielding the invertibility of $\bigl(I-Q^{(2)}_{n+1}\bigr)$. For such values of $n$ by a straightforward calculation one has:
\begin{multline}\label{V-2,R-6}
    B^{(1)}_n=\bigl(I-Q^{(2)}_{n+1}\bigr)^{-1}\bigl(I+Q^{(2)}_n-Q^{(2)}_{n+1}\bigr)\bigl(I-Q^{(2)}_n\bigr)
    +\bigl(I-Q^{(2)}_{n+1}\bigr)^{-1}R^{(5)}_n\bigl(I-Q^{(2)}_n\bigr)
    \\
    =I+\underbrace{\bigl(I-Q^{(2)}_{n+1}\bigr)^{-1}\bigl(Q^{(2)}_{n+1}-Q^{(2)}_n\bigr)Q^{(2)}_n}_{=:V^{(2)}_n}
    +\underbrace{\bigl(I-Q^{(2)}_{n+1}\bigr)^{-1}R^{(5)}_n\bigl(I-Q^{(2)}_n\bigr)}_{=:R^{(6)}_n}.
\end{multline}
Using the trivial bounds $\|V^{(2)}_n\|<2\|V^{(1)}_n\|\|Q^{(2)}_n\|,\|R^{(6)}_n\|<3\|R^{(5)}_n\|,\|V^{(1)}_n\|\le\frac{2\beta}n$ and the bound $\|Q^{(2)}_n(\varepsilon)\|\le\frac{4\beta}{n\left|\sin\frac{\varepsilon}2\right|}$, one has
\begin{equation*}
    \sum_{n>\frac y{|\varepsilon|}}\left\|V^{(2)}_n(\varepsilon)+R^{(6)}_n(\varepsilon)\right\|
    \le\sum_{n>\frac y{|\varepsilon|}}\frac{16\beta^2|\varepsilon|}{y\left|\sin\frac{\varepsilon}2\right|}+3\sum_{n>\frac y{|\varepsilon|}}\left\|R^{(5)}_n(\varepsilon)\right\|
    \rightarrow0\text{ as }y\rightarrow+\infty
\end{equation*}
uniformly in $\varepsilon\in U\backslash\{0\}$. A rather rough argument repeating the scalar estimates yields:
\begin{equation*}
    \left\|\left(\prod_{n>\frac y{|\varepsilon|}}B^{(1)}_n(\varepsilon)\right)-I\right\|
    \le\exp\left(\sum_{n>\frac y{|\varepsilon|}}\|V^{(2)}_n(\varepsilon)+R^{(6)}_n(\varepsilon)\|\right),
\end{equation*}
hence the assertion of the lemma holds, if $B_n$ is replaced by $B^{(1)}_n$. Then, coming back to the product of matrices $B_n(\varepsilon)$:
\begin{equation}\label{previous equality}
    \prod_{k>\frac y{|\varepsilon|}}B_k(\varepsilon)=\Bigl(I-\lim\limits_{n\rightarrow\infty}Q^{(2)}_n(\varepsilon)\Bigr)
    \left(\prod_{k>\frac y{|\varepsilon|}}B^{(1)}_k(\varepsilon)\right)
    \Bigl(I-Q^{(2)}_{\left\lfloor\frac{y}{|\varepsilon|}\right\rfloor+1}(\varepsilon)\Bigr)^{-1}.
\end{equation}
Due to the estimate \eqref{estimate of the sum of V-1}, $Q^{(2)}_{\left\lfloor\frac{y}{|\varepsilon|}\right\rfloor}(\varepsilon)\rightarrow0$ as $y\rightarrow+\infty$ uniformly in $\varepsilon\in U\backslash\{0\}$. Therefore the convergence to the identity matrix in \eqref{previous equality} is uniform. This completes the proof.
\end{proof}

The next lemma completes Step I and proves the uniform boundedness of all partial products of matrices $B_n(\varepsilon)$. This lemma is rather non-trivial and plays an important role in the rest of the proof. It is only due to the choice of the scaling factor in $v_n=\exp\left(\beta\int_1^n\frac{\cos(\varepsilon r)}rdr\right)u_n$ that these products are uniformly bounded.

    \begin{lem}\label{lem a priori estimate}
    Let $\{R_n(\varepsilon)\}_{n=1}^{\infty}\in l^1(U)$. Then there exists a constant $c_3$ such that for every $\varepsilon\in U$ and every $n\le\infty$
    \begin{equation}\label{a priori estimate}
        \left\|\prod_{k=1}^nB_k(\varepsilon)\right\|<c_3.
    \end{equation}
    \end{lem}

\begin{proof}
Using the decomposition of the exponent \eqref{decomposition of the exponent} again one can rewrite the sequence $B_n$ in the following form:
\begin{equation*}
    B_n(\varepsilon)=
    \left(%
    \begin{array}{cc}
    1 & 0 \\
    0 & \exp\l(-2\beta\int\limits_n^{n+1}\frac{\cos(\varepsilon r)}rdr\r) \\
    \end{array}%
    \right)
    +V^{(3)}_n(\varepsilon)+R^{(7)}_n(\varepsilon)
\end{equation*}
with some $\{R^{(7)}_n(\varepsilon)\}_{n=1}^{\infty}\in l^1(U)$ and
\begin{multline}\label{V^{(3)}}
    V^{(3)}_n(\varepsilon):=\frac{\beta}n\int_n^{n+1}\cos(\varepsilon r)dr
    \left(
      \begin{array}{cc}
        -1 & 0 \\
        0 & 1 \\
      \end{array}
    \right)
    +\frac{\beta}n
    \left(
      \begin{array}{cc}
        \cos(\varepsilon n) & \sin(\varepsilon n) \\
        \sin(\varepsilon n) & -\cos(\varepsilon n) \\
      \end{array}
    \right)
    \\
    =\frac{\beta}n\left(\cos(\varepsilon n)-\int_n^{n+1}\cos(\varepsilon r)dr\right)
    \left(
      \begin{array}{cc}
        1 & 0 \\
        0 & -1 \\
      \end{array}
    \right)
    +\frac{\beta}n\sin(\varepsilon n)
    \left(
      \begin{array}{cc}
        0 & 1 \\
        1 & 0 \\
      \end{array}
    \right).
\end{multline}
We are going to perform the variation of parameters in the discrete equation
\begin{equation*}
    u_{n+1}=
    \left[
    \left(%
    \begin{array}{cc}
    1 & 0 \\
    0 & \exp\l(-2\beta\int\limits_n^{n+1}\frac{\cos(\varepsilon r)}rdr\r) \\
    \end{array}%
    \right)
    +V^{(3)}_n(\varepsilon)+R^{(7)}_n(\varepsilon)
    \right]
    u_n,
\end{equation*}
considering it as a perturbation of the equation
\begin{equation*}
    u_{n+1}=
    \left(%
    \begin{array}{cc}
    1 & 0 \\
    0 & \exp\l(-2\beta\int\limits_n^{n+1}\frac{\cos(\varepsilon r)}rdr\r) \\
    \end{array}%
    \right)
    u_n.
\end{equation*}
This leads to the following:
\begin{multline}\label{equation for u after the variation of parameters}
    u_n(\varepsilon,f)=
    \left(%
    \begin{array}{cc}
    1 & 0 \\
    0 & \exp\l(-2{\beta}\int_1^n\frac{\cos(\varepsilon r)}rdr\r) \\
    \end{array}%
    \right)
    f
    \\
    +\sum_{k=1}^{n-1}
    \left(%
    \begin{array}{cc}
    1 & 0 \\
    0 & \exp\l(-2{\beta}\int_{k+1}^n\frac{\cos(\varepsilon r)}rdr\r) \\
    \end{array}%
    \right)
    \left(V^{(3)}_k(\varepsilon)+R^{(7)}_k(\varepsilon)\right)u_k(\varepsilon,f).
\end{multline}
Using the Gronwall's lemma and a simple estimate
\begin{equation}\label{estimate of the oscillating integral}
    \exp\l(-2\beta \int_t^y\frac{\cos s}sds\r)
    \le
    \exp\l(-2\beta \int_{\frac{\pi}2}^{\frac{3\pi}2}\frac{\cos s}sds\r)
    \le
    3^{2\beta},\text{ if }0\le t\le y\le\infty,
\end{equation}
one gets the following bound for the solution.
\begin{equation*}
    \|u_n(\varepsilon,f)\|
    \le3^{2\beta}\exp\l(3^{2\beta}\sum_{k=1}^{n-1}\left\|V^{(3)}_k(\varepsilon)+R^{(7)}_k(\varepsilon)\right\|\r)\|f\|.
\end{equation*}
In other terms,
\begin{equation}\label{estimate of the product}
    \left\|\prod_{k=1}^nB_k(\varepsilon)\right\|\le3^{3\beta}\exp\left(3^{3\beta}
    \sum_{k=1}^n\left\|V^{(3)}_k(\varepsilon)+R^{(7)}_k(\varepsilon)\right\|\right).
\end{equation}
Furthermore, the identity \eqref{V^{(3)}} leads to the following estimate:
\begin{equation*}
    \left\|V^{(3)}_n(\varepsilon)\right\|\le\frac{5\beta}2|\varepsilon|\text{ for every }n\in\mathbb N\text{ and }\varepsilon\in U,
\end{equation*}
which can be easily obtained with the help of the equality \eqref{equality for the integral} and explicit bounds $\left|\frac{\cos\varepsilon-1}{\varepsilon}\right|\le1$,
$1-\frac{\sin\varepsilon}{\varepsilon}\le\frac{|\varepsilon|}2$.
Thus, if $n\le\frac{c_2}{|\varepsilon|}$ (where $c_2$ is defined by \eqref{c-2}), then
\begin{equation}\label{sum of V-3+R-7}
    \sum_{k=1}^n\left\|V^{(3)}_k(\varepsilon)+R^{(7)}(\varepsilon)\right\|\le\frac{5\beta c_2}2+\sum_{k=1}^{\infty}\left\|R^{(7)}(\varepsilon)\right\|
\end{equation}
including the case of $\varepsilon=0, n=\infty$.

The estimate \eqref{estimate of the product} grows to infinity with $n$, therefore the tail of the matrix product ought to be considered separately. If $\varepsilon\neq0$ and $n>\frac{c_2}{|\varepsilon|}$, one obtains from the equality $B^{(1)}_n=I+V^{(2)}_n+R^{(6)}_n$, see \eqref{V-2,R-6}:
\begin{equation*}
    \left\|\prod_{k=\left\lfloor\frac{c_2}{|\varepsilon|}\right\rfloor+1}^nB^{(1)}_k(\varepsilon)\right\|
    <\exp\left(\sum_{k=\left\lfloor\frac{c_2}{|\varepsilon|}\right\rfloor+1}^n
    \left\|V^{(2)}_k(\varepsilon)+R^{(6)}_k(\varepsilon)\right\|\right).
\end{equation*}
Estimating the exponent in the same way as in the previous lemma one gets:
\begin{equation}\label{sum of V-2+R-6}
    \sum_{k=\left\lfloor\frac{c_2}{|\varepsilon|}\right\rfloor+1}^n
    \left\|V^{(2)}_k(\varepsilon)+R^{(6)}_k(\varepsilon)\right\|
    <\frac{16\beta^2|\varepsilon|}{c_2\left|\sin\frac{\varepsilon}2\right|}
    +3\sum_{k=1}^{\infty}\left\|R^{(5)}_k(\varepsilon)\right\|.
\end{equation}
Since both expressions \eqref{sum of V-3+R-7} and \eqref{sum of V-2+R-6} are bounded uniformly with respect to $\varepsilon$, $B_n=\bigl(I-Q^{(2)}_{n+1}\bigr)B^{(1)}_n\bigl(I-Q^{(2)}_n\bigr)^{-1}$ and $\|Q^{(2)}_n\|<1/2$, the product $\prod_{k=\left\lfloor\frac{c_2}{|\varepsilon|}\right\rfloor+1}^nB_k(\varepsilon)$ is bounded for $n>\frac{c_2}{|\varepsilon|}$ uniformly with respect to $\varepsilon$. Assertion of the lemma follows.
\end{proof}

\subsection*{Step II. Rewriting the system $u_{n+1}=B_n(\varepsilon)u_n$ in the form of a Volterra integral equation in the slow scale}
Consider the equation \eqref{equation for u after the variation of parameters},
\begin{multline*}
    u_n(\varepsilon,f)=
    \left(%
    \begin{array}{cc}
    1 & 0 \\
    0 & \exp\l(-2{\beta}\int_1^n\frac{\cos(\varepsilon r)}rdr\r) \\
    \end{array}%
    \right)
    f
    \\
    +\sum_{k=1}^{n-1}
    \left(%
    \begin{array}{cc}
    1 & 0 \\
    0 & \exp\l(-2{\beta}\int_{k+1}^n\frac{\cos(\varepsilon r)}rdr\r) \\
    \end{array}%
    \right)
    \left(V^{(3)}_k(\varepsilon)+R^{(7)}_k(\varepsilon)\right)u_k(\varepsilon,f).
\end{multline*}
which is equivalent to the system $u_{n+1}=B_n(\varepsilon)u_n$. On this step we rewrite it in an integral operator form. Fix an $\varepsilon\in U\backslash\{0\}$. Put again $n=\left\lfloor\frac y{|\varepsilon|}\right\rfloor$ and divide the sum in \eqref{equation for u after the variation of parameters} into two sums, which contain terms $R^{(7)}$ and $V^{(3)}$, respectively. Then write the second sum of the two as an integral in a new variable $\tau$ putting $k=\lfloor\tau\rfloor$. Doing this we get for $y\ge|\varepsilon|$, since a piece-wise constant function appears in the second sum:
\begin{multline}\label{equation for u with integral}
    u_{\left\lfloor\frac y{|\varepsilon|}\right\rfloor}(\varepsilon)=
    \left(%
    \begin{array}{cc}
    1 & 0 \\
    0 & \exp\l(-2{\beta}\int_1^{\left\lfloor\frac y{|\varepsilon|}\right\rfloor}\frac{\cos(\varepsilon r)}rdr\r) \\
    \end{array}%
    \right)
    f
    \\
    +\sum_{k=1}^{\left\lfloor\frac y{|\varepsilon|}\right\rfloor-1}
    \left(%
    \begin{array}{cc}
    1 & 0 \\
    0 & \exp\l(-2{\beta}\int_{k+1}^{\left\lfloor\frac y{|\varepsilon|}\right\rfloor}\frac{\cos(\varepsilon r)}rdr\r) \\
    \end{array}%
    \right)
    R^{(7)}_k(\varepsilon)u_k(\varepsilon)
    \\
    +\int_1^{\left\lfloor\frac y{|\varepsilon|}\right\rfloor}
    \left(%
    \begin{array}{cc}
    1 & 0 \\
    0 & \exp\l(-2{\beta}\int_{\lfloor\tau\rfloor+1}^{\left\lfloor\frac y{|\varepsilon|}\right\rfloor}\frac{\cos(\varepsilon r)}rdr\r) \\
    \end{array}%
    \right)
    V^{(3)}_{\lfloor\tau\rfloor}(\varepsilon)u_{\lfloor\tau\rfloor}(\varepsilon)d\tau.
\end{multline}
Scaling the variable of the integration $\tau=\frac t{|\varepsilon|}$ we write this as a Volterra integral equation:
\begin{equation}\label{equation for z}
    z(y)=g(y)+\int_0^yK(y,t)z(t)dt.
\end{equation}
Here we denote by $z,g$ and $K$ the following piecewise-constant (on intervals of length $\varepsilon$) functions:
\begin{equation*}
    z(y,\varepsilon,f):=u_{\left\lfloor\frac y{|\varepsilon| }\right\rfloor}(\varepsilon,f),\text{ if }y\ge|\varepsilon|,
\end{equation*}
\begin{multline}\label{g}
    g(y,\varepsilon,f):=
    \left(%
    \begin{array}{cc}
    1 & 0 \\
    0 & \exp\l(-2{\beta}\int_{|\varepsilon|}^{|\varepsilon|\left\lfloor\frac y{|\varepsilon|}\right\rfloor}\frac{\cos s}sds\r) \\
    \end{array}%
    \right)
    f
    \\
    +
    \sum_{k=1}^{\left\lfloor\frac y{|\varepsilon|}\right\rfloor-1}
    \left(%
    \begin{array}{cc}
    1 & 0 \\
    0 & \exp\l(-2{\beta}\int_{k+1}^{\left\lfloor\frac y{|\varepsilon|}\right\rfloor}\frac{\cos(\varepsilon r)}rdr\r) \\
    \end{array}%
    \right)
    R^{(7)}_k(\varepsilon)u_k(\varepsilon,f),\text{ if }y\ge|\varepsilon|,
\end{multline}
\begin{equation*}
    K(y,t,\varepsilon):=
    \left(%
    \begin{array}{cc}
    1 & 0\\
    0 & \exp\l(-2{\beta}
    \int_{|\varepsilon|\l(\left\lfloor\frac t{|\varepsilon|}\right\rfloor+1\r)}^{|\varepsilon|\left\lfloor\frac y{|\varepsilon|}\right\rfloor}\frac{\cos s}sds\r)\\
    \end{array}%
    \right)
    \frac{V^{(3)}_{\left\lfloor\frac t{|\varepsilon|}\right\rfloor}(\varepsilon)}{|\varepsilon|},
    \text{ if }|\varepsilon|\le t<|\varepsilon|\left\lfloor\frac y{|\varepsilon|}\right\rfloor.
\end{equation*}
For all the other values of $y$ (and $t$, if it is present) define these functions to be equal zero. After we have done this, we can successfully use standard operator methods. Before doing that let us observe the point-wise convergence of the kernel $K$ and the "free term" $g$ of the integral equation \eqref{equation for z}. It is easy to see from the definition
\begin{equation*}
    V^{(3)}_n(\varepsilon)=\frac{\beta}n\int_n^{n+1}\cos(\varepsilon r)dr
    \left(
      \begin{array}{cc}
        -1 & 0 \\
        0 & 1 \\
      \end{array}
    \right)
    +\frac{\beta}n
    \left(
      \begin{array}{cc}
        \cos(\varepsilon n) & \sin(\varepsilon n) \\
        \sin(\varepsilon n) & -\cos(\varepsilon n) \\
      \end{array}
    \right)
\end{equation*}
that
\begin{equation*}
    \frac{V^{(3)}_{\left\lfloor\frac y{|\varepsilon|}\right\rfloor}(\varepsilon)}{\varepsilon}\rightarrow\frac{\beta\sin y}y
    \left(
      \begin{array}{cc}
        0 & 1 \\
        1 & 0 \\
      \end{array}
    \right)
    \text{ as }\varepsilon\rightarrow0
\end{equation*}
and hence for $y>t>0$
\begin{equation}\label{convergence of the kernel}
    K(y,t,\varepsilon)\rightarrow\pm
    \left(
      \begin{array}{cc}
        0 & 1 \\
        \exp\left(-2\beta\int_t^y\frac{\cos s}sds\right) & 0 \\
      \end{array}
    \right)
    \frac{\beta\sin t}t\text{ as }\varepsilon\rightarrow\pm0.
\end{equation}
Additionally, for every $y>0$
\begin{equation}\label{convergence of the free term}
    g(y,\varepsilon)\rightarrow
    \left(%
    \begin{array}{cc}
    1 & 0 \\
    0 & 0 \\
    \end{array}%
    \right)
    \l(f+\sum_{k=1}^{\infty}R^{(7)}_k(0)u_k(0,f)\r)\text{ as }\varepsilon\rightarrow0.
\end{equation}
This follows from the uniform bound given by Lemma \ref{lem a priori estimate} and from the property
\linebreak
$\{R^{(7)}_n(\varepsilon)\}_{n=1}^{\infty}\in l^1(U)$ by Lebesgue dominated convergence theorem.

    \begin{rem}\label{rem very important}
    It is very important  that the limit in \eqref{convergence of the free term} coincides with the $\lim\limits_{n\rightarrow\infty}u_n(0,f)$. The latter can be obtained from the following variant of discrete Levinson theorem \cite[Lemma 4.4, case (b)]{Janas-Simonov-2010}.
    \end{rem}

    \begin{prop}[\cite{Janas-Simonov-2010}]\label{proposition from Janas-Simonov}
    Suppose that $\sum\limits_{k=1}^{\infty}\frac{\|\hat R_k\|}{\left|\hat\lambda_k\right|}<\infty$ and that there exist an $M$ such that for every $m\ge n$ the estimate $\prod\limits_{l=n+1}^m|\hat\lambda_l|\ge\frac1M$ holds. Let $\hat u$ be a solution of the system
    \begin{equation*}
        \hat u_{n+1}=\left[
        \left(%
            \begin{array}{cc}
            \hat\lambda_n & 0 \\
            0 & 1/{\hat\lambda_n} \\
            \end{array}%
        \right)
        +\hat R_n\right]\hat u_n.
    \end{equation*}
    If $\prod\limits_{l=1}^{\infty}|\hat\lambda_l|=\infty$, then both sides of the following equality exist and the equality holds:
    \begin{equation*}
        \lim\limits_{n\rightarrow\infty}
        \frac{\hat u_n}{\prod\limits_{l=1}^{n-1}\hat\lambda_l}
        =
        \left(%
            \begin{array}{cc}
            1 & 0 \\
            0 & 0 \\
            \end{array}%
        \right)
        \left[
        \hat u_1+\sum\limits_{k=1}^{\infty}
        \frac{\hat R_k\hat u_k}{\prod\limits_{l=1}^k\hat\lambda_l}\right].
    \end{equation*}
    \end{prop}

In the case of our system $u_{n+1}=B_n(\varepsilon)u_n$, one can successfully apply this argument, putting $\hat u_n=n^{\beta}u_n(0,f)$, $\hat\lambda_n=\left(\frac{n+1}n\right)^{\beta}$ and $\hat R_n=\left(\frac{n+1}n\right)^{\beta}R^{(7)}_n(0)$.

\subsection*{Step III. Convergence in the slow scale}
Now we establish the convergence as $\varepsilon\rightarrow\pm0$ of the solution $u(\varepsilon)$ in the slow scale. We remind that in our notation $u_{\left\lfloor\frac y{|\varepsilon|}\right\rfloor}(\varepsilon)=z(y,\varepsilon)$.

    \begin{lem}\label{lem convergence in the slow scale}
    For every $y>0$ and $f\in\mathbb C^2$ there exist two limits
    \begin{equation*}
        h_{\pm}(y,f):=\lim_{\varepsilon\rightarrow\pm0}z(y,\varepsilon,f),
    \end{equation*}
    which satisfy the following integral equations:
    \begin{equation}\label{equations for h-pm}
        h_{\pm}(y,f)=\lim_{n\rightarrow\infty}u_n(0,f)
        \pm\int_0^y
        \left(
          \begin{array}{cc}
            0 & 1 \\
            \exp\left(-2\beta\int_t^y\frac{\cos s}sds\right) & 0 \\
          \end{array}
        \right)
        \frac{\beta\sin t}th_{\pm}(t,f)dt.
    \end{equation}
    \end{lem}

\begin{proof}
For every $y_0>0$ and $\varepsilon\neq0$ define the operator $\mathcal K_{y_0}(\varepsilon)$ in the Banach space $L_{\infty}((0;y_0),\mathbb C^2)$ by the rule
\begin{equation}\label{K-y-0(epsilon)}
    \mathcal K_{y_0}(\varepsilon): u(y)\mapsto\int_0^yK(y,t,\varepsilon)u(t)dt,\ y\in(0;y_0).
\end{equation}
Denote $K(y,t,\pm0):=\lim\limits_{\varepsilon\rightarrow\pm0}K(y,t,\varepsilon)$ and analogously define two operators $\mathcal K_{y_0}(\pm0)$. We consider only the case $\varepsilon\rightarrow+0$ here. The second case can be treated in the same way.

First let us see that the operator $\mathcal K_{y_0}(\varepsilon)$ converges to $\mathcal K_{y_0}(+0)$ as $\varepsilon\rightarrow+0$ in the space of bounded linear operators in $L_{\infty}((0;y_0),\mathbb C^2)$. It is enough to show that
\begin{equation}\label{what we really need to prove}
    \max_{y\in[0;y_0]}\int_0^y\|K(y,t,\varepsilon)-K(y,t,+0)\|dt\rightarrow0\text{ as }\varepsilon\rightarrow+0.
\end{equation}
Fix a $\Delta>0$. Since both kernels are uniformly bounded in all variables, there exists a $y_1(\Delta)<y_0$ such that for every positive $\varepsilon\in U$
\begin{equation*}
    \max_{y\in[0;y_1(\Delta)]}\int_0^y\|K(y,t,\varepsilon)-K(y,t,+0)\|dt<\Delta.
\end{equation*}
For the same reason
\begin{multline}\label{cropping the epsilon}
    \max_{y\in[y_1(\Delta);y_0]}\int_0^y\|K(y,t,\varepsilon)-K(y,t,+0)\|dt
    \\
    =\max_{y\in[y_1(\Delta);y_0]}\int_{\varepsilon}^{y-\varepsilon}\|K(y,t,\varepsilon)-K(y,t,+0)\|dt
    +O(\varepsilon)\text{ as }\varepsilon\rightarrow+0.
\end{multline}
Using the estimates $\frac{1-\cos\varepsilon}{\varepsilon}=O(\varepsilon)$ and $1-\frac{\sin\varepsilon}{\varepsilon}=O(\varepsilon^2)$ as $\varepsilon\rightarrow0$, one can write
\begin{equation*}
    K(y,t,\varepsilon)=
    \left(
      \begin{array}{cc}
        0 & 1 \\
        \exp\l(-2{\beta}
        \int_{\varepsilon\l(\left\lfloor\frac t{\varepsilon}\right\rfloor+1\r)}^{\varepsilon\left\lfloor\frac y{\varepsilon}\right\rfloor}\frac{\cos s}sds\r) & 0 \\
      \end{array}
    \right)
    \frac{\beta\sin\left(\varepsilon\left\lfloor\frac t{\varepsilon}\right\rfloor\right)}{\varepsilon\left\lfloor\frac t{\varepsilon}\right\rfloor}+O(\varepsilon)\text{ as }\varepsilon\rightarrow+0
\end{equation*}
for $y$ and $t$ such that $2\varepsilon<y<y_0$ and $\varepsilon<t<y-\varepsilon$, where $O(\varepsilon)$ is uniform with respect to these $y$ and $t$. We remind that
\begin{equation*}
    K(y,t,+0)=
    \left(
      \begin{array}{cc}
        0 & 1 \\
        \exp\left(-2\beta\int_t^y\frac{\cos s}sds\right) & 0 \\
      \end{array}
    \right)
    \frac{\beta\sin t}t.
\end{equation*}
The mapping $(y;t)\mapsto\exp\left(-2\beta\int_t^y\frac{\cos sds}s\right)$ is uniformly continuous on the compact set $y_1(\Delta)\le y\le y_0$, $0\le t\le y$ (notice that it has a discontinuity at the point $y=t=0$), the function $t\mapsto\frac{\sin t}t$ is uniformly continuous in the interval $[0;y_0]$. Therefore
\linebreak
$\max_{t\in[\varepsilon;y-\varepsilon]}\|K(y,t,\varepsilon)-K(y,t,+0)\|=o(1)$ as $\varepsilon\rightarrow+0$ uniformly in $y\in[y_1(\Delta);y_0]$. Together with \eqref{cropping the epsilon} this means that for sufficiently small positive values of $\varepsilon$ the maximum
\begin{equation*}
    \max_{y\in[y_1(\Delta);y_0]}\int_0^y\|K(y,t,\varepsilon)-K(y,t,+0)\|dt<\Delta,
\end{equation*}
and hence the maximum in \eqref{what we really need to prove} has the same property. Since $\Delta$ is an arbitraty positive number, the convergence of operators $\mathcal K_{y_0}(\varepsilon)$ follows.

Secondly, as mentioned in Remark \ref{rem very important}, for every $y>0$ one has:
\begin{equation*}
    g(y,\varepsilon)\rightarrow\lim_{n\rightarrow\infty}u_n(0)\text{ as }\varepsilon\rightarrow0.
\end{equation*}
We face an obstacle here: this convergence is only point-wise and is not in the norm of $L_{\infty}((0;y_0),\mathbb C^2)$. The same is true for the convergence of $z(y,\varepsilon)$, which we are going to prove. There is no sense in directly applying the inverse of the operator $I-\mathcal K_{y_0}(\varepsilon)$. The proper object to consider is the difference $z(y,\varepsilon)-g(y,\varepsilon)$, which converges in $L_{\infty}((0;y_0),\mathbb C^2)$. To see this let us rewrite the equation \eqref{equation for z} in the following form:
\begin{equation}\label{equation for z proper form}
    z(\varepsilon)-g(\varepsilon)=(I-\mathcal K_{y_0}(\varepsilon))^{-1}\mathcal K_{y_0}(\varepsilon)g(\varepsilon)
\end{equation}
(since $\mathcal K_{y_0}(\varepsilon)$ is a Volterra integral operator, $(I-\mathcal K_{y_0}(\varepsilon))^{-1}$ exists). It suffices to prove that $K_{y_0}(+0)g(\varepsilon)\rightarrow K_{y_0}(+0)\lim_{n\rightarrow\infty}u_n(0)$ in $L_{\infty}((0;y_0),\mathbb C^2)$. By a direct estimate we have:
\begin{equation*}
    \max_{y\in[0;y_0]}\int_0^y
    \left\|K(y,t,+0)\left(g(t,\varepsilon)-\lim_{n\rightarrow\infty}u_n(0)\right)\right\|dt
    \le
    \beta3^{2\beta}\int_0^{y_0}\|g(t,\varepsilon)-\lim_{n\rightarrow\infty}u_n(0)\|dt.
\end{equation*}
The right-hand side tends to zero as $\varepsilon\rightarrow+0$ due to the point-wise convergence and the uniform boundedness of $g(t,\varepsilon)$ by Lebesgue dominated convergence theorem. From \eqref{equation for z proper form} we obtain:
\begin{equation*}
    z(\varepsilon)-g(\varepsilon)\rightarrow(I-\mathcal K_{y_0}(+0))^{-1}\mathcal K_{y_0}(+0)\lim_{n\rightarrow\infty}u_n(0)\text{ as }\varepsilon\rightarrow+0\text{ in }L_{\infty}((0;y_0),\mathbb C^2)
\end{equation*}
From the point-wise convergence of $g(y,\varepsilon)$ it follows that for every $y>0$
\begin{equation*}
    z(y,\varepsilon)\rightarrow
    \left(\left(I+(I-\mathcal K_{y_0}(+0))^{-1}\mathcal K_{y_0}(+0)\right)\lim_{n\rightarrow\infty}u_n(0)\right)(y)=:h_+(y).
\end{equation*}
Consider the integral equation \eqref{equation for z} again. Taking the limit as $\varepsilon\rightarrow+0$ turns it into the equation for $h_+$, \eqref{equations for h-pm}. This is possible due to the uniform boundedness of the solution $z$ and of the kernel $K$ by Lebesgue dominated convergence theorem. This completes the proof.
\end{proof}

    \begin{rem}
    The assertion of the last lemma includes the case when $\lim_{n\rightarrow\infty}u_n(0,f)=0$. In this case $h_{\pm}(y,f)\equiv0$. This happens only for one particular direction of the vector $f\in\mathbb C^2$, because the rank of $\Theta$ is one.
    \end{rem}

\subsection*{Step IV. Convergence as $y\rightarrow+\infty$ and changing the order of limits}
It remains to show that every solution of integral equations \eqref{equations for h-pm} has a limit at infinity and to see that it is possible to change the order of limits in
$\lim\limits_{y\rightarrow+\infty}\lim\limits_{\varepsilon\rightarrow\pm0}
u_{\left\lfloor\frac y{|\varepsilon|}\right\rfloor}(\varepsilon,f)$.

    \begin{lem}\label{lem convergence of h-pm}
    The following two limits exist: $\lim\limits_{y\rightarrow+\infty}h_{\pm}(y,f)$. These limits are non-zero, iff $\lim\limits_{n\rightarrow\infty}u_n(0,f)\neq0$.
    \end{lem}

\begin{proof}
Consider the integral equation for $h_+$:
\begin{equation*}
    h_+(y)=\lim\limits_{n\rightarrow\infty}u_n(0,f)
    +\int_0^y
    \left(%
    \begin{array}{cc}
    0 & 1 \\
    \exp\l(-2\beta\int_t^y\frac{\cos s}sds\r) & 0 \\
    \end{array}%
    \right)
    \frac{\beta\sin t}th_+(t)dt.
\end{equation*}
Let as write it as a differential equation using the following factorization property of the matrix in the integral:
\begin{multline*}
    \left(%
    \begin{array}{cc}
    0 & 1 \\
    \exp\l(-2\beta\int_t^y\frac{\cos s}sds\r) & 0 \\
    \end{array}%
    \right)
    =
    \left(%
    \begin{array}{cc}
    1 & 0 \\
    0 & \exp\l(2\beta\int_y^{\infty}\frac{\cos s}sds\r)
    \end{array}%
    \right)
    \\
    \times
    \left(%
    \begin{array}{cc}
    0 & 1 \\
    \exp\l(-2\beta\int_t^{\infty}\frac{\cos s}sds\r) & 0 \\
    \end{array}%
    \right).
\end{multline*}
Multiplying from the left by
$\left(%
\begin{array}{cc}
1 & 0 \\
0 & \exp\l(2\beta\int_y^{\infty}\frac{\cos s}sds\r)
\end{array}%
\right)^{-1}$
and taking the derivative in $y$ one has:
\begin{equation*}
    h_+'(y)-\frac{\beta}y
    \left(
      \begin{array}{cc}
        0 & \sin y \\
        \sin y & -2\cos y \\
      \end{array}
    \right)
    h_+(y)
    =\frac{2\beta\cos y}y
    \left(
      \begin{array}{cc}
        0 & 0 \\
        0 & 1 \\
      \end{array}
    \right)
    \lim\limits_{n\rightarrow\infty}u_n(0,f).
\end{equation*}
By Lemma \ref{lem convergence of the product and degeneracy} the vector $\lim\limits_{n\rightarrow\infty}u_n(0,f)\in\mathbb C^2$ is either proportional to
$\left(
   \begin{array}{c}
     1 \\
     0 \\
   \end{array}
\right)$ or is equal to zero and therefore
\begin{equation}\label{differential equation for h+ with no right hand side}
    h_+'(y)=\frac{\beta}y
    \left(
      \begin{array}{cc}
        0 & \sin y \\
        \sin y & -2\cos y \\
      \end{array}
    \right)
    h_+(y).
\end{equation}
If $\lim\limits_{n\rightarrow\infty}u_n(0,f)\neq0$, then $h_+(y)$ is not identically zero. Every non-zero solution of \eqref{differential equation for h+ with no right hand side} has a non-zero limit as $y\rightarrow+\infty$, which follows for example from Harris-Lutz results \cite[Theorem 3.1, p.85]{Harris-Lutz-1975} (one should take there
$\Lambda(y)=-\frac{2\beta\cos y}y
\left(
  \begin{array}{cc}
    0 & 0 \\
    0 & 1 \\
  \end{array}
\right)$,
$V(y)=\frac{\beta\sin y}y
\left(
  \begin{array}{cc}
    0 & 1 \\
    1 & 0 \\
  \end{array}
\right)$,
$R(y)=0$). For the second limit (of $h_-(y,f)$) the proof is absolutely analogous.
\end{proof}

The following lemma is a combination of definitions and previous results: Lemmas \ref{lem convergence of the tail}, \ref{lem convergence in the slow scale}, \ref{lem convergence of h-pm}.

    \begin{lem}\label{lem equality of the limits}
    The following two limits exist: $\lim\limits_{\varepsilon\rightarrow\pm0}\lim\limits_{n\rightarrow\infty}u_n(\varepsilon,f)$ and equal to
    \linebreak
    $\lim\limits_{y\rightarrow+\infty}h_{\pm}(y,f)$, respectively.
    \end{lem}

\begin{proof}
By definition of $h_{\pm}$,
$\lim\limits_{y\rightarrow+\infty}h_{\pm}(y,f)
=\lim\limits_{y\rightarrow+\infty}\lim\limits_{\varepsilon\rightarrow\pm0}
u_{\left\lfloor\frac y{|\varepsilon|}\right\rfloor}(\varepsilon,f)$.
By Lemma \ref{lem convergence of the tail},
\linebreak
$u_{\left\lfloor\frac y{|\varepsilon|}\right\rfloor}(\varepsilon,f)$
converges uniformly with respect to $\varepsilon\in U\backslash\{0\}$. Therefore there exist two limits,
$\lim\limits_{\varepsilon\rightarrow\pm0}
\lim
\limits_{y\rightarrow+\infty}u_{\left\lfloor\frac y{|\varepsilon|}\right\rfloor}(\varepsilon,f)$
which coincide with the limits
$\lim\limits_{y\rightarrow+\infty}\lim\limits_{\varepsilon\rightarrow\pm0}
u_{\left\lfloor\frac y{|\varepsilon|}\right\rfloor}(\varepsilon,f)$,
respectively. Clearly
$\lim\limits_{y\rightarrow+\infty}
u_{\left\lfloor\frac y{|\varepsilon|}\right\rfloor}(\varepsilon,f)
=
\lim\limits_{n\rightarrow\infty}u_n(\varepsilon,f)$
for $\varepsilon\neq0$, which implies that $\lim\limits_{\varepsilon\rightarrow\pm0}\lim\limits_{n\rightarrow\infty}u_n(\varepsilon,f)
=
\lim\limits_{y\rightarrow+\infty}\lim\limits_{\varepsilon\rightarrow\pm0}
u_{\left\lfloor\frac y{|\varepsilon|}\right\rfloor}(\varepsilon,f)
=
\lim\limits_{y\rightarrow+\infty}h_{\pm}(y,f)$.
This completes the proof.
\end{proof}

We are ready now to finish the proof of Theorem \ref{thm discrete-continuous}. Limits $\lim\limits_{\varepsilon\rightarrow\pm0}u_{\left\lfloor\frac y{|\varepsilon|}\right\rfloor}(\varepsilon,f)$ exist by Lemma \ref{lem convergence in the slow scale}, the equality $\lim\limits_{\varepsilon\rightarrow\pm0}\lim\limits_{n\rightarrow\infty}u_n(\varepsilon,f)
=\lim\limits_{y\rightarrow+\infty}h_{\pm}(y,f)$ has sense and holds true by Lemma \ref{lem equality of the limits}. The map $\Theta$ has the rank one by Lemma \ref{lem convergence of the product and degeneracy} (see also Remark \ref{rem about Theta}).
\end{proof}

\section{Zeroes of the spectral density}\label{section result}
In this section, we prove the main result of the paper by applying Theorem \ref{thm discrete-continuous}.

    \begin{thm}\label{thm result}
    Let $q_1\in L_1(\mathbb R_+)$, $\omega\notin\frac{\pi\mathbb Z}{2a}$, and $\rho'_{\alpha}(\lambda)$ be the spectral density of the operator $\mathcal L_{\alpha}$. Let the index $j$ be greater or equal to zero. If the solution $\varphi_{\alpha}(y,\nu_{j,+})$ of \eqref{phi} is not a subordinate one, then there exist two non-zero limits
    \begin{equation*}
        \lim_{\lambda\rightarrow\nu_{j,+}\pm0}\frac{\rho'_{\alpha}(\lambda)}{|\lambda-\nu_{j,+}|^{\frac{2|c|}{a|W\{\psi_+(\nu_{j,+}),\psi_-(\nu_{j,+})\}|}
        \l|\int\limits_0^a\psi^2_+(t,\nu_{j,+})e^{2i\omega t}dt\r|}}.
    \end{equation*}
    Analogously, if $\varphi_{\alpha}(y,\nu_{j,-})$ is not subordinate, then there exist two non-zero limits
    \begin{equation*}
        \lim_{\lambda\rightarrow\nu_{j,-}\pm0}\frac{\rho'_{\alpha}(\lambda)}{|\lambda-\nu_{j,-}|^{\frac{2|c|}
        {a|W\{\psi_+(\nu_{j,+}),\psi_-(\nu_{j,+})\}|}
        \l|\int\limits_0^a\psi^2_-(t,\nu_{j,-})e^{2i\omega t}dt\r|}}.
    \end{equation*}
    \end{thm}

\begin{proof}
We use the notation of Section \ref{section discretization}: $\nu_{cr}$, $\beta$ instead of $\nu_{j,\pm}$, $\beta_{j,\pm}$. This will yield both claims of the theorem. By \eqref{v-alpha}, the solution $\varphi_{\alpha}(y,\lambda)$ is related to the solution $\{v_{\alpha,n}(\varepsilon)\}_{n=1}^{\infty}$
of \eqref{system for v with epsilon}. In turn, the latter is related by the scaling
\linebreak
$v_n=\exp\left(\beta\int_1^n\frac{\cos(\varepsilon r)}rdr\right)u_n$
to the solution
$\{u_n(\varepsilon,v_{\alpha,1}(\varepsilon))\}_{n=1}^{\infty}$
of the system $u_{n+1}=B_n(\varepsilon)u_n$. We now apply Theorem to the sequence \ref{thm discrete-continuous} to
$\{u_n(\varepsilon,v_{\alpha,1}(\varepsilon))\}_{n=1}^{\infty}$.
Non-subordinacy of $\varphi_{\alpha}(\nu_{cr})$ means that
$\lim\limits_{n\rightarrow\infty}u_n(0,v_{\alpha,1}(0))\neq0$.
Due to continuity of $v_{1,\alpha}(\cdot)$ in $U$, Theorem \ref{thm discrete-continuous} implies the existence of the following two limits
\begin{equation*}
    \lim_{\varepsilon\rightarrow\pm0}\lim_{\ninf}u_n(\varepsilon,v_{\alpha,1}(\varepsilon))\neq0.
\end{equation*}
Taking into account the scaling $v_n=\exp\left(\beta\int_1^n\frac{\cos(\varepsilon r)}rdr\right)u_n$ this means that there exist two non-zero limits $\lim\limits_{\varepsilon\rightarrow\pm0}\lim\limits_{n\rightarrow\infty}|\varepsilon|^{\beta}v_{\alpha,n}(\varepsilon)$. This in turn yields the assertion of the theorem due to the Weyl-Titchmarsh type formula \eqref{Weyl-Titchmarsh type formula v} and the property of the quasi-momentum that its derivative is positive inside spectral bands.
\end{proof}

\section{Discussion}\label{section discussion}

As mentioned in Introduction, our result is to be compared to the
one of Hinton-Klaus-Shaw \cite{Hinton-Klaus-Shaw-1991}. The
difference in the classes of operators considered may seem
significant: in \cite{Hinton-Klaus-Shaw-1991} there is no periodic
background and the non-summable part of the potential is an
infinite sum of Wigner-von Neumann type terms. Nevertheless, both
problems can be written after a few suitable transformations in
virtually the same form, see system \eqref{system for eta} in our
case and the formula prior to (2.2) in
\cite{Hinton-Klaus-Shaw-1991}. We are also able to consider finite
or even infinite sum of Wigner-von Neumann type terms and to
reduce the problem to the same form by analogous transformations.
The methods that we use also have common traits like usage of the
slow variable $y=|\varepsilon|n=|\varepsilon|x/a$ or work with the
limit integral Volterra equation like \eqref{equations for h}
which actually represents an equation of the same type as (2.75)
in \cite{Hinton-Klaus-Shaw-1991} but is written in a different, more
explicit form. However, the approach that we develop in the
present paper to our mind has several essential distinctions compared to the
approach of \cite{Hinton-Klaus-Shaw-1991}.

1. We use a discretization of the differential system and  work
with the discrete model system of a ``simple'' form \eqref{system
for v with epsilon}. In addition to bringing certain technical
(inessential) difficulties it makes our method more universal and enables one to
consider a wider class of problems. As an example, we can analyze by
a direct application of Theorem \ref{thm discrete-continuous} the
structure of zeroes of the spectral density of the discrete
Schr\"odinger operator with Wigner-von Neumann potential. This is
by far not the only possible example. Using the solution of the
model problem one can also consider differential Schr\"odinger
operators with point interactions at integer points with strengths
of interaction forming a sequence of the Wigner-von Neumann form.

2. We do not reduce the problem to a scalar one  and use the
analysis of vector equations instead. This makes the analysis
simpler and more transparent. As a consequence of this, the form
of the limit equation \eqref{equations for h} is explicit and can
be guessed from heuristic considerations (see \eqref{system for u
form with the difference}-\eqref{equations for h-pm
differential}).

3. We avoid fine technically involved estimates of   oscillatory
integrals and in fact our proof is divided into two parts. First
(Step I), we prove rather rough a priori estimates by dividing the
positive half-line into subintervals $\mathbb
R_+=\bigl(0;\frac{c_2}{|\varepsilon|}\bigr]\cup\bigl(\frac{c_2}{|\varepsilon|};+\infty\bigr)$
for  sufficiently large $c_2$. Second (Steps II-IV), we deal with
the slow scale and prove the convergence using operator
techniques. This is the main distinction of our method which makes
it sufficiently simpler than the method of
\cite{Hinton-Klaus-Shaw-1991}.

We expect  that using the technique developed in the present paper
one can treat a more complicated case of the Wigner-von Neumann
perturbation of non-Coulomb type with slowly decaying power part.
We plan to address this problem in the future.

\section*{Acknowledgements}
Authors wish to express their gratitude to Dr. Roman Romanov for
helpful remarks and fruitful discussions and to Dr. Alexander V.
Kiselev for his help with the work on the text and many useful
remarks and comments. The work was supported by the grant
RFBR-09-01-00515-a. The second author was also supported by the Chebyshev Laboratory
(Department of Mathematics and Mechanics, Saint-Petersburg State
University) under the grant 11.G34.31.2006 of the Government of the
Russian Federation.


\begin{thebibliography}{10}

\bibitem{Aronszajn-1957}
N.~Aronszajn.
\newblock {On a problem of Weyl in the theory of singular Sturm-Liouville
  equations}.
\newblock {\em Amer. J. Math.}, 79(3):597--610, 1957.

\bibitem{Behncke-1991-I}
H.~Behncke.
\newblock {Absolute continuity of Hamiltonians with von Neumann Wigner
  potentials I}.
\newblock {\em Proc. Amer. Math. Soc.}, 111:373--384, 1991.

\bibitem{Behncke-1991-II}
H.~Behncke.
\newblock {Absolute continuity of Hamiltonians with von Neumann Wigner
  potentials II}.
\newblock {\em Manuscripta Math.}, 71(1):163--181, 1991.

\bibitem{Behncke-1994}
H.~Behncke.
\newblock {The m-function for Hamiltonians with Wigner-von Neumann potentials}.
\newblock {\em J. Math. Phys.}, 35(4):1445--1462, 1994.

\bibitem{Benzaid-Lutz-1987}
Z.~Benzaid and D.A. Lutz.
\newblock {Asymptotic representation of solutions of perturbed systems of
  linear difference equations}.
\newblock {\em Stud. Appl. Math.}, 77(3):195--221, 1987.

\bibitem{Brown-Eastham-McCormack-1998}
B.M. Brown, M.S.P. Eastham, and D.K.R. McCormack.
\newblock {Absolute continuity and spectral concentration for slowly decaying
  potentials}.
\newblock {\em J. Comput. Appl. Math.}, 94:181--197, 1998.
\newblock arXiv:math/9805025v1.

\bibitem{Buslaev-Matveev-1970}
V.S. Buslaev and V.B. Matveev.
\newblock {Wave operators for the Schr{\"o}dinger equation with a slowly
  decreasing potential}.
\newblock {\em Theoret. Math. Phys.}, 2(3):266--274, 1970.

\bibitem{Buslaev-Skriganov-1974}
V.S. Buslaev and M.M. Skriganov.
\newblock {Coordinate asymptotic behavior of the solution of the scattering
  problem for the Schr\"odinger equation}.
\newblock {\em Theoret. Math. Phys.}, 19(2):465--476, 1974.

\bibitem{Capasso-et-al-1992}
F.~Capasso, C.~Sirtori, J.~Faist, D.L. Sivco, S.N.G. Chu, and A.Y. Cho.
\newblock {Observation of an electronic bound state above a potential well}.
\newblock {\em Nature}, 358:565--567, 1992.

\bibitem{Herbst-2002}
J.~Cruz-Sampedro, I.~Herbst, and R.~Martinez-Avendano.
\newblock {Perturbations of the Wigner-von Neumann potential leaving the
  embedded eigenvalue fixed}.
\newblock 3(2):331--345, 2002.

\bibitem{Damanik-Simon-2006}
D.~Damanik and B.~Simon.
\newblock {Jost functions and Jost solutions for Jacobi matrices, I. A
  necessary and sufficient condition for Szeg{\H{o}} asymptotics}.
\newblock {\em Invent. Math.}, 165(1):1--50, 2006.
\newblock arXiv:math/0502486v1.

\bibitem{Deift-Killip-1999}
P.~Deift and R.~Killip.
\newblock {On the Absolutely Continuous Spectrum of One-Dimensional
  Schr{\"o}dinger Operators with Square Summable Potentials}.
\newblock {\em Comm. Math. Phys.}, 203(2):341--347, 1999.

\bibitem{Gilbert-Pearson-1987}
D.J. Gilbert and D.B. Pearson.
\newblock {On subordinacy and analysis of the spectrum of one-dimensional
  Schr\"odinger operators}.
\newblock {\em J. Math. Anal. Appl.}, 128(1):30--56, 1987.

\bibitem{Harris-Lutz-1975}
W.A. Harris and D.A Lutz.
\newblock {Asymptotic integration of adiabatic oscillators}.
\newblock {\em J. Math. Anal. Appl.}, 51:76--93, 1975.

\bibitem{Hinton-Klaus-Shaw-1991}
D.B. Hinton, M.~Klaus, and J.K. Shaw.
\newblock {Embedded half-bound states for potentials of Wigner-von Neumann
  type}.
\newblock {\em Proc. Lond. Math. Soc.}, 3(3):607--646, 1991.

\bibitem{Janas-Simonov-2010}
J.~Janas and S.~Simonov.
\newblock {Weyl-Titchmarsh type formula for discrete Schr\"odinger operator
  with Wigner-von Neumann potential}.
\newblock {\em Studia Math.}, 201(2):167--189, 2010.
\newblock arXiv:1003.3319, mp\_arc 10-47.

\bibitem{Kiselev-1996}
A.~Kiselev.
\newblock {Absolutely continuous spectrum of one-dimensional Schr{\"o}dinger
  operators and Jacobi matrices with slowly decreasing potentials}.
\newblock {\em Comm. Math. Phys.}, 179(2):377--399, 1996.

\bibitem{Klaus-1991}
M.~Klaus.
\newblock {Asymptotic behavior of Jost functions near resonance points for
  Wigner-von Neumann type potentials}.
\newblock {\em J. Math. Phys.}, 32:163--174, 1991.

\bibitem{Kodaira-1949}
K.~Kodaira.
\newblock {The eigenvalue problem for ordinary differential equations of the
  second order and Heisenberg's theory of $S$-matrices}.
\newblock {\em Amer. J. Math.}, 71(4):921--945, 1949.

\bibitem{Kurasov-1992}
P.~Kurasov.
\newblock {Zero-range potentials with internal structures and the inverse
  scattering problem}.
\newblock {\em Lett. Math. Phys.}, 25(4):287--297, 1992.

\bibitem{Kurasov-1996}
P.~Kurasov.
\newblock {Scattering matrices with finite phase shift and the inverse
  scattering problem}.
\newblock {\em Inverse Problems}, 12(3):295--307, 1996.

\bibitem{Kurasov-Naboko-2007}
P.~Kurasov and S.~Naboko.
\newblock {Wigner-von Neumann perturbations of a periodic potential: spectral
  singularities in bands}.
\newblock {\em Math. Proc. Cambridge Philos. Soc.}, 142(01):161--183, 2007.

\bibitem{Kurasov-Simonov-2011}
P.~Kurasov and S.~Simonov.
\newblock {Weyl-Titchmarsh type formula for periodic Schr\"odinger operator
  with Wigner-von Neumann potential}.
\newblock {\em Preprints in Mathematical Sciences, Lund University}, 6:1--26,
  2010.

\bibitem{Matveev-1973}
V.B. Matveev.
\newblock {Wave operators and positive eigenvalues for a Schr{\"o}dinger
  equation with oscillating potential}.
\newblock {\em Theoret. Math. Phys.}, 15(3):574--583, 1973.

\bibitem{Naboko-1986}
S.N. Naboko.
\newblock {Dense point spectra of Schr{\"o}dinger and Dirac operators}.
\newblock {\em Theoret. Math. Phys.}, 68(1):646--653, 1986.

\bibitem{Nesterov-2007}
P.N. Nesterov.
\newblock Averaging method in the asymptotic integration problem for systems
  with oscillatory-decreasing coefficients.
\newblock {\em Differ. Equ.}, 43(6):745--756, 2007.

\bibitem{Simon-1997}
B.~Simon.
\newblock {Some Schr{\"o}dinger operators with dense point spectrum}.
\newblock {\em Proc. Amer. Math. Soc.}, 125(1):203--208, 1997.

\bibitem{Titchmarsh-1946-1}
E.C. Titchmarsh.
\newblock {\em {Eigenfunction expansions associated with second-order
  differential equations. Part I}}.
\newblock Clarendon Press, Oxford, 1946.

\bibitem{Titchmarsh-1946-2}
E.C. Titchmarsh.
\newblock {\em {Eigenfunction expansions associated with second-order
  differential equations. Part II}}.
\newblock Clarendon Press, Oxford, 1946.

\bibitem{Wigner-von-Neumann-1929}
J.~von Neumann and E.P. Wigner.
\newblock {{\"U}ber merkw{\"u}rdige diskrete Eigenwerte}.
\newblock {\em Z. Phys.}, 30:465--467, 1929.

\end{thebibliography}
\end{document}